\UseRawInputEncoding
\documentclass[11pt,twoside]{article}

\usepackage{fullpage, amsthm, amsmath, amssymb, hyperref, dsfont, mathtools,bm}
\usepackage{framed}
\usepackage{caption,subcaption,graphicx,rotating,multirow}
\usepackage{url}
\usepackage{color,xcolor}
\usepackage{enumitem}
\usepackage{epstopdf}

\theoremstyle{plain}
\newtheorem{theorem}{Theorem}
\newtheorem{definition}[theorem]{Definition}

\newtheorem*{remark}{Remark}
\newtheorem{lemma}[theorem]{Lemma}

\newtheorem{corollary}[theorem]{Corollary}
\newtheorem{conjecture}{Conjecture}

\newtheorem{proposition}[theorem]{Proposition}

\usepackage[mathscr]{euscript}
\DeclareSymbolFont{rsfs}{U}{rsfs}{m}{n}
\DeclareSymbolFontAlphabet{\mathscrsfs}{rsfs}

\newcommand{\one}{\mathbf{1}}
\def\Par{{\sf P}}
\newcommand{\rmd}{\mathrm{d}}
\newcommand{\x}{\mathbf{x}}

\renewcommand{\u}{\bm{u}}
\newcommand{\bu}{\bm{u}}
\renewcommand{\v}{\bm{v}}

\newcommand{\m}{\bm{m}}

\newcommand{\N}{\mathbb{N}}

\newcommand{\R}{\mathbb{R}}

\newcommand{\rmS}{\mathrm{S}}
\newcommand{\sphere}{\mathbb{S}^{N-1}(\sqrt{N})}
\newcommand{\bg}{\bm{g}}
\newcommand{\bsigma}{\bm{\sigma}}
\newcommand{\blambda}{\bm{\lambda}}
\newcommand{\bA}{\bm{A}}
\newcommand{\bG}{\bm{G}}
\newcommand{\bM}{\bm{M}}

\DeclareMathOperator*{\plim}{p-lim}
\DeclareMathOperator{\var}{var}

\renewcommand{\P}{\operatorname{\mathbb{P}}}
\newcommand{\E}{\operatorname{\mathbb{E}}}
\newcommand{\argmin}{\operatorname{argmin}}

\newcommand{\normal}{\ensuremath{\mathcal{N}}}
\newcommand{\indi}{\mathds{1}}
\newcommand{\mF}{\mathcal{F}}
\newcommand{\mA}{\mathcal{A}}
\newcommand{\mN}{\mathcal{N}}

\def\cuU{\mathscrsfs{U}}

\def\RS{\mbox{\rm RS}}
\def\sRS{\mbox{\rm\tiny RS}}

\def\RSB{\text{\rm\small RSB}}
\def\FRSB{\text{\rm\small FRSB}}

\def\SF{{\sf SF}}
\def\PDE{\text{\rm\small PDE}}
\def\SATUNSAT{\text{\rm\small SAT-UNSAT}}
\def\sSAT{\text{\rm\tiny SAT}}
\def\sG{\text{\rm\tiny G}}
\def\AMP{\text{\rm\small AMP}}
\def\IAMP{\text{\rm\small IAMP}}
\def\RSAMP{\text{\rm\small RS-AMP}}
\def\TAP{\text{\rm\small TAP}}
\def\SK{\text{\rm\small SK}}
\newcommand{\eps}{\varepsilon}
\newcommand{\ubq}{\bar{q}}
\newcommand{\lbq}{\underline{q}}
\newcommand{\bigo}{\ensuremath{\mathcal{O}}}
\newcommand\smallo{
  \mathchoice
    {{\scriptstyle\mathcal{O}}}
    {{\scriptstyle\mathcal{O}}}
    {{\scriptscriptstyle\mathcal{O}}}
    {\scalebox{.7}{$\scriptscriptstyle\mathcal{O}$}}
  }

\numberwithin{equation}{section}

\begin{document}

\title{
{\bf{\LARGE{Algorithmic pure states for the negative spherical perceptron}}}
}

\author{Ahmed {El Alaoui}\thanks{The Simons Institute for the Theory of Computing, UC Berkeley.} \and Mark Sellke\thanks{Department of Mathematics, Stanford University.}}

\date{}
\maketitle

\vspace*{-.3in} 
\begin{abstract}
We consider the spherical perceptron with Gaussian disorder. This is the set $S$ of points $\bsigma \in \R^N$ on the sphere of radius $\sqrt{N}$ satisfying $\langle \bg_a , \bsigma \rangle \ge \kappa\sqrt{N}\,$ for all $1 \le a \le M$, where $(\bg_a)_{a=1}^M$ are independent standard gaussian vectors  and $\kappa \in \R$ is fixed. Various characteristics of $S$ such as its surface measure and the largest $M$ for which it is non-empty, were computed heuristically in statistical physics in the asymptotic regime $N \to \infty$, $M/N \to \alpha$. The case $\kappa<0$ is of special interest as $S$ is conjectured to exhibit a hierarchical tree-like geometry known as \emph{full replica-symmetry breaking} ($\FRSB$) close to the satisfiability threshold $\alpha_{\sSAT}(\kappa)$, and whose characteristics are captured by a Parisi variational principle akin to the one appearing in the Sherrington-Kirkpatrick model. 
In this paper we design an efficient algorithm which, given oracle access to the solution of the Parisi variational principle, exploits this conjectured $\FRSB$ structure for $\kappa<0$ and outputs a vector $\hat{\bsigma}$ satisfying $\langle \bg_a , \hat{\bsigma}\rangle \ge \kappa \sqrt{N}$ for all $1\le a \le M$ and lying on a sphere of non-trivial radius $\sqrt{\ubq N}$, where $\ubq \in (0,1)$ is the right-end of the support of the associated Parisi measure. We expect $\hat{\bsigma}$ to be approximately the barycenter of a \emph{pure state} of the spherical perceptron. Moreover we expect that $\ubq \to 1$ as $\alpha  \to \alpha_{\sSAT}(\kappa)$, so that $\big\langle \bg_a,\frac{\hat{\bsigma}}{|\hat{\bsigma}|}\big\rangle \geq (\kappa-o(1))\sqrt{N}$ near criticality.
\end{abstract}

\section{Introduction and main result}
Let $\bG$ be an $M\times N$ matrix with i.i.d.\ standard Gaussian entries. Denote its rows by $(\bg_{a})_{a=1}^{M}$, and let $\kappa \in \R$ be a fixed parameter. We consider the problem of finding a vector $\bsigma \in \sphere$ on the sphere of radius $\sqrt{N}$ such that 
\begin{equation}\label{eq:perceptron}
\bG\bsigma \ge \kappa \sqrt{N} \one. 
\end{equation} 
The above inequality is interpreted entrywise: $\langle \bg_a , \bsigma \rangle  \ge \kappa \sqrt{N}$ for all $1 \le a \le M$.
We refer to this as the spherical perceptron problem with Gaussian disorder.  
Let \[\rmS_{M,N}(\bG) := \left\{\bsigma \in \sphere ~:~ \bG\bsigma \ge \kappa \sqrt{N} \one \right\}\] be the set of solutions
and let $M_N$ be the largest value of the number $M$ of hyperplanes for which $\rmS_{M,N}$ is non-empty:
\[M_N : = \max\big\{M \in \mathbb{N} ~:~ \rmS_{M,N}(\bG) \neq \varnothing\big\}.\]  
This is a random variable which depends on the disorder matrix $\bG$. 
The ratio \[\alpha_N := \frac{M_{N}}{N}\] is known as \emph{the storage capacity} of the perceptron. 
The spherical perceptron problem appeared in the statistical physics literature in the late 1980's relating to the early theory on neural networks. It is also closely related to questions in discrepancy theory where typically, the inequalities~\eqref{eq:perceptron} are replaced by $|\bG\bsigma| \le \kappa \sqrt{N} \one$, the matrix $\bG$ is possibly non-random, and one searches for a binary solutions $\bsigma \in \{\pm 1\}^N$, see e.g.,~\cite{spencer1985six,bansal2010constructive,rothvoss2017constructive,bansal2019line,alweiss2020discrepancy,turner2020balancing} and references therein. 
           
Gardner~\cite{gardner1988space} famously computed the typical measure of $\rmS_{M,N}(\bG)$ in logarithmic scale, and determined the asymptotic behavior of $\alpha_N$ when $\kappa \ge 0$ using the physicists' replica method. Further results based both on replica and cavity methods were soon after obtained in~\cite{gardner1988optimal,mezard1989space}.  
Gardner's formula was rigorously confirmed in three mathematical treatments~\cite{shcherbina2003rigorous,talagrand2011mean2,stojnic2013another}, each one using a different technique. Furthermore, the algorithmic question of \emph{finding} a solution essentially reduces to a convex optimization problem when $\kappa\ge 0$. Indeed, the constraint $\|\bsigma\|_2 = \sqrt{N}$ can be relaxed to $\|\bsigma\|_2 \le \sqrt{N}$ after which we obtain a convex feasibility problem with linear inequalities. This problem can be solved to arbitrary accuracy with an interior point method, which returns a solution $\bsigma^*$ in the unit ball whenever the polytope $\{\bsigma \in \R^N : \bG\bsigma \ge \kappa \sqrt{N} \one \}$ has non-empty interior (see e.g., the monograph~\cite{BoydVa04}). Then the vector $\bsigma^*$ can be scaled up to have norm exactly $\sqrt{N}$. Since $\kappa \ge 0$ the newly obtained vector still satisfies the linear inequalities.   

The case $\kappa<0$ is much less explored and will be the main focus of this paper. This case has only been recently studied in~\cite{jones2020spherical} in the worst-case regime where $\bG$ has arbitrary rows of norm $\sqrt{N}$. The authors propose an efficient algorithm for finding a vector $\bsigma \in \rmS_{M,N}(\bG)$ when $M \ge 16 N$ and $\kappa \simeq -\sqrt{2\log(M/N)}$. Returning to the random case, our main result is an efficient algorithm which, for a certain set of `admissible' values of $\kappa$ and the aspect ratio $M/N$, takes the matrix $\bG$ and the parameter $\kappa$ as input, and outputs a vector $\hat{\bsigma}$ satisfying the inequalities~\eqref{eq:perceptron} and lying on a sphere of non-trivial radius $\sqrt{\ubq N}$ with $\ubq \in (0,1)$. The set of admissible values $(\kappa,M/N)$ together with the radius $\ubq$ are described by a certain variational principle which is conjectured to be deeply related to the geometry of the set of solutions $\rmS_{M,N}(\bG)$. Further, we expect that for large $N$, $\ubq \to 1$ as $M/N$ gets close to the storage capacity $\alpha_N$. 
In order to be more precise we must review some of the literature on this subject. 

It was known at the physics level of rigor since the work of Gardner and Derrida~\cite{gardner1988optimal} that the solution obtained in~\cite{gardner1988space} for $\kappa\ge 0$, dubbed ``replica-symmetric", cannot extend to large negative values of $\kappa$, and that a more sophisticated ansatz must be adopted to compute the surface measure of $\rmS_{M,N}(\bG)$ together with the storage capacity $\alpha_N$.
This expectation was partially confirmed mathematically by Stojnic~\cite{stojnic2013another} who showed that $\alpha_N$ is upper-bounded by Gardner's prediction with high probability as $N\to \infty$ for \emph{all} values of $\kappa \in \R$, and in a subsequent treatment~\cite{stojnic2013negative}, he proved an upper bound which is presumably sharper for $\kappa$ large enough in the negative direction, albeit this latter claim was argued via numerical evaluation. 
The spherical perceptron with negative $\kappa$ has been the object of renewed interest from the physics community due to its connection with the phenomenon of jamming of hard spheres in finite dimension, where it is for instance used as a mean-field approximation for the computation of various critical exponents close to the so-called jamming transition~\cite{charbonneau2014exact,franz2016simplest,franz2019critical,franz2020critical}.     
In~\cite{franz2017universality}, Franz and coauthors managed to describe the full phase diagram in the $(\kappa,\alpha)$ plane of the structure of the set of solutions $\rmS_{M,N}(\bG)$. (Here, one considers the `proportional' regime where $M/N \to \alpha$.) It is expected that the structure of $\rmS_{M,N}(\bG)$ undergoes several phase transitions as $\alpha$ is varied while $\kappa$ is fixed at a negative value, the last transition being the $\SATUNSAT$, or \emph{jamming}, transition at the (putative) critical capacity $\alpha_{\sSAT}(\kappa)  := \lim \alpha_N(\kappa)$; see~\cite[Figure 3]{franz2017universality}. This capacity $\alpha_{\sSAT}(\kappa)$ can be characterized in terms of a Parisi variational principle, akin to the one appearing in other mean-field spin glass problems such as the Sherrington-Kirkpatrick ($\SK$) model.
Moreover, it is predicted that the $\SATUNSAT$ line $(\kappa,\alpha_{\sSAT}(\kappa))$ borders a sector $\tilde{\Lambda} = \{(\kappa,\alpha) \in \R_- \times \R_+ : \alpha_{\sG}(\kappa) < \alpha < \alpha_{\sSAT}(\kappa)\}$ in the $(\kappa,\alpha)$ plane where the model has ``continuous" or ``full replica-symmetry breaking" ($\FRSB$). (The threshold $\alpha_{\sG}(\kappa)$ marks the so-called \emph{Gardner transition} from ``$1\RSB$" to $\FRSB$; we do not go into further detail as this transition will be irrelevant to our goal.)  
Mathematically, $\FRSB$ is the property that the support of the Parisi measure produced by the variational principle is an interval $[\lbq,\ubq]$. This is also commonly referred to as a \emph{no overlap gap} condition and will be formally defined in the next section. We are now ready to state out main theorem.              

 \begin{theorem} \label{thm:main}
There exists a set $\Lambda \subset \R_- \times \R_+$ and $\ubq = \ubq(\gamma_*) \in (0,1)$, both defined in Definition~\ref{def:nogap} such that the following holds. Let $(\kappa,\alpha) \in \Lambda$, $\varepsilon>0$, and $M = \lfloor \alpha N\rfloor$. Then there exists an algorithm which takes the matrix $\bG$ and the parameter $\kappa$ as input, runs in time $C(\varepsilon)MN$, and outputs a vector $\bsigma \in \R^N$ with the following two properties holding with probability converging to one as $N \to \infty$:
\[ \left|\frac{\|\bsigma\|_2}{\sqrt{\ubq N}} - 1 \right| \le \varepsilon~~~\mbox{and}~~~ \Big\| \Big(\kappa \one - \frac{1}{\sqrt{N}}\bG\bsigma \Big)_+ \Big\|_2 \le \varepsilon \sqrt{N}.\]  
\end{theorem}
Note that given $\bsigma$ (and under the mild condition $\alpha \neq 1$) one can extract a vector $\hat{\bsigma} \in \R^N$ with either one of the following two properties:
\begin{align}
\big|\|\hat{{\bsigma}}\|_2 - \sqrt{\ubq N}\big| &\le c_0\varepsilon \sqrt{\ubq N}~~~~~ \mbox{and} ~~~~~ \bG\hat{{\bsigma}} \ge \kappa \sqrt{N} \one \, ,\label{eq:round1} \\
\mbox{or}~~~~~\|\hat{{\bsigma}}\|_2 &= \sqrt{\ubq N}~~~~~ \mbox{and}~~~~~\bG\hat{{\bsigma}} \ge  \frac{\kappa}{1-c_0\varepsilon} \sqrt{N} \one \, , \label{eq:round2}
\end{align}
where $c_0 = c_0(\alpha)>0$.
To obtain~\eqref{eq:round1}, we project $\bsigma$ over the set of linear constraints~\eqref{eq:perceptron}. To obtain~\eqref{eq:round2}, we additionally project on the sphere of radius $\sqrt{\ubq N}$.  These projection steps can be performed for instance via a cutting plane method at the price of $C(\varepsilon)N^3(\log N)^c$ arithmetic operations with $c>0$ a universal constant for an $\varepsilon$-accurate solution. We refer to Section~\ref{sec:rounding} for the precise statements.
  
We now make a few clarifying remarks. 

  The set $\Lambda$ is defined in Section~\ref{sec:Parisi_var} as the set of points for which the model is $\FRSB$, and is expected to take the form of a sector $\tilde{\Lambda} = \{(\kappa,\alpha) \in \R_- \times \R_+ : \alpha_{\sG}(\kappa) < \alpha < \alpha_{\sSAT}(\kappa)\}$ as described above. 
We do not prove that the set $\Lambda$ is non-empty, therefore our result is conditional in nature. As we explain in the next section, $\FRSB$ is expected but remains an unsolved conjecture in mathematical physics. A similar conjecture pertaining to the case of the Sherrington-Kirkpatrick model has been recently used in~\cite{mon18} to produce an algorithm returning a point on the hypercube approximately maximizing the random quadratic $\langle \bsigma , \bA \bsigma \rangle$, where $\bA$ is a GOE matrix.  

The radius $\ubq = \ubq(\kappa,\alpha)$ has a special significance; it is the right-end of the support of the Parisi measure for the spherical perceptron (see next section), and the following picture is expected. Letting $\mu_N$ denote the surface measure on the sphere $\mathbb S^{N-1}(\sqrt{N})$, one can decompose the set of solutions $\rmS_{M,N}(\bG)$ into a random countable disjoint union $(\cup_a C_a) \cup R$ such that $\mu_N(R)=\smallo(\mu_N(\rmS_{M,N}(\bG)))$, and for each $a$, the subset $C_a$---called a ``pure state"---has a barycenter at radius $\sqrt{\ubq N}$, and the uniform distribution over $C_a$ concentrates in the sense that the inner product between two independently and uniformly chosen points from $C_a$ normalized by $N$ converges to a deterministic limit equal to $\ubq$. The vector $\bsigma$ (or $\hat{\bsigma}$ as above) output by the algorithm is expected to be (approximately) the barycenter of one pure state $C_a$. In light of this, the problem of producing a genuine solution lying on the sphere of radius $\sqrt{N}$ seems to reduce to \emph{sampling} a point from the pure state $C_a$ whose barycenter is computed by our algorithm, and this seems to require new ideas.

 Our approach is based on an iterative scheme which leverages the structure of the Parisi variational principle in order to construct a feasible solution. This scheme belongs to the larger family of approximate message passing ($\AMP$) algorithms, which notably includes Bolthausen's iterative construction of the $\TAP$ equation for the $\SK$ model~\cite{bolthausen2014iterative}. This class of algorithms has been intensively studied and systematically enlarged by numerous authors; see~\cite{BM-MPCS-2011,javanmard2013state,bayati2015universality, berthier2019state} and references therein.  A recent innovation in this class came in the form of \emph{incremental} $\AMP$ ($\IAMP$) which attempts to exploit the hierarchical organization of states in spin glass models having full replica symmetry breaking~\cite{mon18,ams20}. This scheme was in turn inspired by an algorithm of Subag~\cite{subag2018following} for optimizing spherical mixed $p$-spin models. 
 For the purpose of exhibiting a solution at large radius, our algorithm has three stages. The first stage consists applying a `vanilla' (replica-symmetric) $\AMP$ iteration, similar to Bolhausen's construction~\cite{bolthausen2014iterative}. The second stage is an implementation of $\IAMP$ exploiting the assumed $\FRSB$ structure of the problem in the set $\Lambda$, and the third stage is a rounding step based on convex optimization. We refer to Section~\ref{sec:alg} for a full description and analysis of the algorithm.  

Let us further mention that $\AMP$ was also recently used in the work of Ding and Sun~\cite{ding2019capacity} for proving a lower bound on the capacity of the Ising perceptron model, where one searches for a solution on the binary hypercube $\bsigma \in \{\pm 1\}^N$ instead of the sphere. Heuristic work for the Ising perceptron, analogous to the work of Gardner on the sphere, was performed by Krauth and M\'ezard~\cite{krauth1989storage}, and leads to a different prediction.
Ding and Sun's approach is based on a conditional second moment method, where the $\AMP$ iteration is used to provide an appropriate  $\sigma$-field on which to condition the moment computations, and hence match the Krauth-M\'ezard prediction. The flavor of $\AMP$ used in their work is akin to the first stage of our algorithm.  
It is an interesting question to design efficient algorithms producing solutions in the case of the Ising perceptron, (or conversely, to use our iteration combined with the Ding-Sun approach to prove a capacity lower bound for the spherical case.) We refer to~\cite{aubin2019storage} for an interesting discussion of this algorithmic question and its link to the so-called ``frozen $1\RSB$" structure of $\rmS_{M,N}(\bG)$.

\section{The Parisi variational formula}
 \label{sec:Parisi_var}
 We now formally define the Parisi variational principle, provide its physical interpretation, and define the no overlap gap assumption.
 
 We define the space of functional  order parameters (f.o.p.) $\cuU$ to be the collection of all non-decreasing right-continuous functions $\gamma: [0,1] \to [0,1]$ such that $\gamma(q) = 1$ for some $q<1$.
We define the left- and right-end of the support as 
\[\lbq(\gamma) = \inf\{t: \gamma(t)>0\},~~~ \mbox{and}~~~\ubq(\gamma) = \sup\{t: \gamma(t)<1\}.\]
We can identify such $\gamma$ with cumulative functions of probability measures on $[0,1]$ whose support ends strictly before $1$. 
For $\gamma \in \cuU$, we consider the Parisi $\PDE$
\begin{align}\label{eq:Parisi_pde}
\begin{split}
\partial_t \Phi_{\gamma}(t,x)  +\frac{1}{2}\Big(\gamma(t) (\partial_x \Phi_\gamma(t,x))^2 +  \partial^2_{x}\Phi_\gamma(t,x)\Big) &=0~~\mbox{for all}~ (t,x) \in [0,1)\times \R\\
\Phi_{\gamma}(1,x) &= 
\begin{cases}
0 & \mbox{if } x \ge \kappa ,\\
-\infty & \mbox{otherwise.}  
\end{cases}
\end{split}
\end{align}
Adopting the convention $e^{-\infty}=0$, the condition $\ubq(\gamma)<1$  ensures that the above $\PDE$ has a well defined solution on $[\ubq(\gamma),1)$. Indeed since $\gamma(t)=1$ for $t \in [\ubq(\gamma),1]$, we can use the Cole-Hopf transform to write
 \begin{equation}\label{eq:boundary}
\Phi_{\gamma}(t,x) =  \log \P\big(x + \sqrt{1-t} Z \ge \kappa\big) = \log \mN\left(\frac{\kappa - x}{\sqrt{1-t}}\right),
\end{equation}
for $(t,x) \in [\ubq(\gamma),1) \times \R$, where $Z \sim N(0,1)$ and $\mN(x) = \P(Z \ge x)$. 
In Section~\ref{sec:parisi_pde} we extend this solution to the interval $[0,\ubq(\gamma))$ through a limiting procedure and prove its regularity.   

We define the Parisi functional for the spherical perceptron as
\begin{equation}\label{eq:parisi_functional}
\Par(\gamma) = \alpha\Phi_{\gamma}(0,0) + \frac{1}{2} \int_0^{\ubq} \frac{\rmd q}{\lambda(q)} + \frac{1}{2} \log(1-\ubq),
\end{equation}
where 
\[\lambda(q) =  \int_q^{1} \gamma(t) \rmd t,\] 
and $\ubq = \ubq(\gamma)$. (Note that $\ubq$ can be replaced by any $\bar{\ubq} \ge \ubq$ in formula~\eqref{eq:parisi_functional} with no effect on its value.) 
We denote the infimal value of the Parisi functional on $\cuU$ as 
\begin{equation}\label{eq:parisi_formula}
\Par_{\star}(\alpha,\kappa) = \inf_{\gamma \in \cuU} \Par(\gamma).
\end{equation}
We say that the Parisi formula is \emph{replica-symmetric} when the optimizers of $\Par$ are of the form  $\gamma(t) = \one\{t \ge q\}$ for some $q \in [0,1)$. In this case it follows from formulas~\eqref{eq:boundary} and~\eqref{eq:parisi_functional} that $\Par_{\star}(\alpha,\kappa)$ coincides with 
 \begin{equation}\label{eq:gardner}
\RS(\alpha,\kappa) \equiv \inf_{q \in [0,1)} \left\{\alpha \E \log \mN\left(\frac{\kappa - \sqrt{q}Z}{\sqrt{1-q}}\right) + \frac{1}{2} \frac{q}{1-q} + \frac{1}{2}\log(1-q) \right\},
 \end{equation}
 where $Z \sim N(0,1)$.
 Furthermore, the above infimum is achieved at some $q<1$ if and only if $\alpha < \alpha_{\sRS}(\kappa) $ where   
 \begin{equation}\label{eq:alpha_RS}
\alpha_{\sRS}(\kappa) := \E\big[(\kappa - Z)_+^2\big]^{-1}.
\end{equation}  
(Otherwise the infimum is $-\infty$.) 

Formulas~\eqref{eq:gardner} and~\eqref{eq:alpha_RS} are respectively Gardner's formula and threshold mentioned in the introduction: letting $\mu_N$ be the surface measure on the sphere $\sphere$ normalized to have total mass 1, we have for $\kappa \ge0$ and $\alpha  \ge 0$,
\begin{align*} 
\mathop{\lim_{N \to \infty}}_{M/N \to \alpha} \frac{1}{N} \log \mu_{N}\big( S_{M,N}(\bG) \big) = \RS(\alpha,\kappa)\quad
\mbox{and} \quad \lim_{N \to \infty} \alpha_{N} = \alpha_{\sRS}(\kappa),
\end{align*}
 in probability~\cite{shcherbina2003rigorous,talagrand2011mean2,stojnic2013another}. It is expected that the Parisi formula is indeed replica-symmetric for $\kappa\ge 0$, and that the infimal value in Eq.~\eqref{eq:parisi_formula} generalizes Gardner's formula:   
 \begin{conjecture}[Essentially in~\cite{franz2017universality}] For all $(\kappa,\alpha) \in \R \times \R_+$,
\begin{align*}
\mathop{\lim_{N \to \infty}}_{M/N \to \alpha} \frac{1}{N} \log \mu_{N}\big( S_{M,N}(\bG) \big) =  \Par_\star(\alpha,\kappa) \,, \quad  
\mbox{and}\quad  \lim_{N \to \infty} \alpha_{N} = \alpha_{\sSAT}(\kappa), 
\end{align*}
where 
\[ \alpha_{\sSAT}(\kappa) := \sup\big\{ \alpha \ge 0 \, :\,  \inf_{\gamma \in \cuU} \Par(\gamma)~\mbox{is achieved} \big\} =  \sup\left\{\alpha \ge 0 ~:~ \Par_{\star}(\alpha,\kappa) > -\infty \right\},\]
and the limits hold in probability.
\end{conjecture}

\begin{remark}
We note that $\alpha_{\sSAT}(\kappa) \le \alpha_{\sRS}(\kappa)$ for all $\kappa \in \R$. Indeed, it is not difficult to check that $\Par_{\star}(\alpha,\kappa) = -\infty$ for $\alpha > \alpha_{\sRS}(\kappa)$ by plugging a step function $\gamma(t) = \one\{t \ge q\}$ in $\Par$ and letting $q \to 1^-$. (This can be seen from Gardner's formula Eq.~\eqref{eq:gardner}.) 
\end{remark}

Next we define the no overlap gap condition.
\begin{definition}\label{def:nogap}
We say that the spherical perceptron model at parameters $(\kappa,\alpha)$ has no overlap gap if there exists $\gamma_* \in \cuU$ such that $\lbq(\gamma_*) < \ubq(\gamma_*)$ and $\gamma_*$ is strictly increasing over $[\lbq(\gamma_*),\ubq(\gamma_*)]$, and $\Par(\gamma_*) = \Par_{\star}(\alpha,\kappa)$. Furthermore, we let 
\begin{align}\label{def:Gamma}
\Lambda := \left\{(\kappa,\alpha) \in \R_{-} \times \R_+ ~ \mbox{s.t.\ the model has no overlap gap at}~(\kappa,\alpha)\right\}.
\end{align} 
\end{definition}

As mentioned earlier, no overlap gap is expected to hold in a non-trivial sector of the form $\tilde{\Lambda} = \{(\kappa,\alpha) \in \R_- \times \R_+ : \alpha_{\sG}(\kappa) < \alpha < \alpha_{\sSAT}(\kappa)\}$, and furthermore, $\ubq(\gamma_*) \to 1$ as $\alpha \to \alpha_{\sSAT}(\kappa)$ from below, so our algorithm would produce a vector of norm almost $\sqrt{N}$ when $\alpha$ is close to $\alpha_{\sSAT}$.

The assumption of no overlap gap is known in the physics literature as \emph{continuous} or \emph{full replica-symmetry breaking} ($\FRSB$). Its physical interpretation is as follows: If $\bsigma_1$ and $\bsigma_2$ are two independent and uniform samples from the set of solutions $\rmS_{M,N}(\bG)$ then their normalized overlap $N^{-1}\langle \bsigma_1,\bsigma_2\rangle$ is expected to converge to a random variable with c.d.f.\ $\gamma_*$ as $N \to \infty$, $M/N \to \alpha$. Hence in the replica-symmetric case where  $\gamma_*(t) = \one\{t \ge q_*\}$, $N^{-1}\langle \bsigma_1,\bsigma_2\rangle \to q_*$ in distribution. The condition that $\gamma_*$ is strictly increasing means that the asymptotic support of the overlap distribution does not have a gap. 
This condition has recently emerged as a sufficient condition for the algorithmic tractability of optimizing mean-field spin glass Hamiltonians~\cite{subag2018following,mon18,ams20}  up to an arbitrarily small relative error. For instance, an appropriate version of this condition is widely believed to hold for $\SK$ in the entire low temperature phase, and this belief is supported by numerical experiments~\cite{crisanti2002analysis,schmidt2008method,alaoui2020algorithmic}.

We finish this section by defining an important stochastic process $(X_t)_{t \in [0,1]}$ and stating two main identities implied by the optimality of the minimizer $\gamma_* \in \cuU$ in Definition~\ref{def:nogap}.  
It will be useful throughout to introduce the SDE
\begin{equation}\label{eq:sde}
\rmd X_t  = \gamma_*(t) \partial_x \Phi_{\gamma_*}(t,X_t) \rmd t + \rmd B_t,~~~ X_0 = 0,
\end{equation}
where $(B_t)_{t \in [0,1]}$ is a standard Brownian motion. Existence and uniqueness of a strong solution is established in Section~\ref{sec:analysis_var}. As we will see in Section~\ref{sec:analysis_var}, Proposition~\ref{prop:optimality_gamma}, the optimality of $\gamma_*$ as a minimizer of $\Par$ implies 
\begin{equation}\label{eq:stationary}
\E\big[\partial_x \Phi_{\gamma_*}(t,X_t)^2\big] = \frac{1}{\alpha} \int_0^t \frac{\rmd q}{\lambda(q)^2}
~~~\mbox{and}~~~\alpha\lambda(t)^2\E\big[\partial_x^2 \Phi_{\gamma_*}(t,X_t)^2\big] = 1,
\end{equation}
for all $t \in [\lbq(\gamma_*),\ubq(\gamma_*)]$. Furthermore, letting $t \to \lbq = \lbq(\gamma_*)$ in the first identity, we obtain 
\begin{equation}\label{eq:stationary_q0}
\E\big[\partial_x \Phi_{\gamma_*}(\lbq,X_{\lbq})^2\big] = \frac{\lbq}{\alpha \lambda(\lbq)^2}.
\end{equation}
The above identities will be crucial to the analysis of the algorithm.

\section{The algorithm}
\label{sec:alg}
Our algorithm has three stages. The first two rely on two different flavors of approximate message passing, and the last stage is a convex optimization-based rounding step. 

Let us first describe the second, incremetal step. Our approach will be to construct a sequence of iterates $\bu_{t} \in \R^{N}$ where $t \in \mathbb{T}_{\delta} := \{k\delta: k \ge 0\} \cap [0,1]$ for which we can track the empirical distribution of the entries of $\bu_{t}$ and of the matrix-vector product $\frac{1}{\sqrt{N}}\bG \bu_{t}$.  Here, $\delta>0$ has the interpretation of a time discretization step which will eventually tend to $0$. 
The desired outcome is to simultaneously obtain
\begin{align}\label{eq:result}
 \frac{1}{N}\big\| \u_{\ubq}  \big\|_2^2 = {\ubq}+o_{N,\delta}(1)~~~\mbox{and}~~~
\frac{1}{ \sqrt{M}} \Big\|\Big(\kappa \one - \frac{1}{\sqrt{N}}\bG\u_{\ubq} \Big)_+\Big\|_2 = o_{N,\delta}(1),
\end{align}
where $\ubq = \ubq(\gamma_*)$. (See proposition~\ref{prop:se_combined} for the precise statement.)
We use the technology of $\AMP$ to achieve this. We will see that in the limit $N\to \infty$ and $M/N \to \alpha$ then $\delta \to 0$, the algorithm has a scaling limit where the sequence of vectors $(\bu_t)_{t \in \mathbb{T}_{\delta}}$ converges in a certain sense to a Brownian motion $(\tilde{B}_t)_{t \in [0,1]}$, and for any pseudo-Lipschitz test function $\psi : \R \to \R$, we have
\[\frac{1}{M} \sum_{a=1}^M \psi\big( \langle \bg_a , \u_{t}\rangle /\sqrt{N}\big) \longrightarrow \E\psi\big(\E[X_{1} | \mF_{\ubq}]\big),\]       
where $(X_t)$ is given by the SDE~\eqref{eq:sde}, driven by a Brownian motion $(B_t)_{t \in [0,1]}$ independent of $(\tilde{B}_t)_{t \in [0,1]}$. 
This passage to the scaling limit $\delta \to 0$ crucially relies on the fact that the identities displayed in Eq.~\eqref{eq:stationary} are valid on the full interval $[\lbq,\ubq]$; these identities in turn follow from the assumed no-overlap gap property.
We utilize the properties of the Parisi functional (in particular the terminal condition of the Parisi $\PDE$) to show that $\E[X_1 | \mF_{q}] \ge \kappa$ almost surely whenever the infimum $\inf_{\gamma \in \cuU}\Par(\gamma)$ is achieved, thereby proving~\eqref{eq:result}. 
Once $\bu_{\ubq}$ is computed, we obtain a vector $\hat{\bsigma}$ satisfying~\eqref{eq:round1} or~\eqref{eq:round2} by orthogonal projections.  

Due to the one-sided nature of the perceptron constraints, the sequence $\bu_{t}$ has to be initialized in a non-trivial way at a point $\bu_{\lbq}$ lying approximately on the sphere of radius $\sqrt{\lbq N}$. (Recall that the identities~\eqref{eq:stationary} do not hold on the intervals $[0,\lbq)$ and $(\ubq,1]$.) This will be taken care of with an initial stage where a simpler version of the $\AMP$ algorithm is used. We note that for symmetric versions of the perceptron problem where the inequalities are two-sided, e.g., $|\langle \bg_a , \bsigma \rangle|  \le \kappa \sqrt{N}$, it is the case that $\lbq=0$ and this first stage is not needed as the one can initialize the incremental stage at the origin. 
Alternatively, whenever the Parisi formula is replica-symmetric, i.e., it is minimized at a step function, then the incremental stage is not needed, and one use the simple version of $\AMP$ together with the rounding step.

\subsection{The message passing iteration} 
We consider the general class of Approximate Message Passing ($\AMP$) iterations. 
Let $f_{\ell}, g_{\ell}  : \R^{\ell+1} \to \R$ for all $\ell \ge 0$ be two sequences of real-valued differentiable functions. 
 For a sequence of vectors $\x^0,\cdots,\x^\ell \in \R^d$ we use the notation $f_{\ell}(\x^0,\cdots,\x^\ell)$ for the vector $(f_{\ell}(x_i^0,\cdots,x_i^\ell))_{1\le i \le d}$, i.e., $f_{\ell}$ is applied entrywise. Let $\bA = \frac{1}{\sqrt{N}} \bG$.
The general iteration takes the form 
\begin{align}\label{eq:general_amp}
\begin{split}
\u^{\ell+1} &= \bA^\top f_{\ell}(\v^0,\cdots,\v^\ell)  - \sum_{s=0}^\ell \mathrm{b}_{\ell, s} g_{s}(\u^0,\cdots,\u^{s}),\\  
\v^{\ell} &= \bA \,\, g_{\ell}(\u^0,\cdots,\u^\ell)  - \sum_{s=0}^{\ell} \mathrm{d}_{\ell, s} f_{s-1}(\v^0,\cdots,\v^{s-1}),\\ 
\end{split}
\end{align}  
where
\begin{equation*}
\mathrm{b}_{\ell,s} := \frac{1}{N} \sum_{k=1}^M \frac{\partial f_{\ell}}{\partial v^s}(v_k^0,\cdots,v_k^\ell), \qquad
\mathrm{d}_{\ell,s} := \frac{1}{N} \sum_{i=1}^N \frac{\partial g_{\ell}}{\partial u^s}(u_i^0,\cdots,u_i^\ell).
\end{equation*} 
This iteration is initialized with a vector $\u^0\in\R^N$ with coordinates drawn  i.i.d.\ from a probability distribution $p_0$, independently of everything else.   
The joint distribution for the first $\ell$ iterates of Eq.~\eqref{eq:general_amp} can be exactly characterized in the $N\to \infty$ limit. For $\x \in \R^d$, we let $\langle \x\rangle_d := \frac{1}{d}\sum_{i=1}^d x_i$.

\begin{proposition}[State evolution]\label{prop:state_evolution}  
Assume that $p_0$ has a finite second moment and let $\psi : \R^{\ell+1} \to \R$ be a pseudo-Lipschitz function. Then
\begin{align*}
\Big\langle \psi\big(\u^0,\cdots,\u^\ell\big)\Big\rangle_N &~\xrightarrow[N\to \infty]{p} ~\E\psi\big(U^0,\cdots,U^\ell\big),\\
\Big\langle \psi\big(\v^0,\cdots,\v^\ell\big)\Big\rangle_M &~\xrightarrow[N\to \infty]{p} ~\E\psi\big(V^0,\cdots,V^\ell\big),
\end{align*}  
as $N \to \infty, M/N \to \alpha$.
Here, $U_0 \sim p_0$, $(U^1,\cdots,U^\ell)$ and $(V^0,\cdots,V^\ell)$ are centered Gaussian vectors, independent of $U^0$ and of each other. Their covariances are defined recursively by 
\begin{align}\label{eq:cov_amp}
\begin{split}
\E\big[U^{\ell+1}U^{j+1}\big] &= \alpha \E\big[f_{\ell}(V^0,\cdots,V^\ell) f_{j}(V^0,\cdots,V^j)\big],\\
\E\big[V^{\ell}V^{j}\big] &=  ~~\E\big[g_{\ell}(U^0,\cdots,U^\ell) g_{j}(U^0,\cdots,U^j)\big],
\end{split}
~~ \mbox{for}~~\ell,j \ge 0.
\end{align}
\end{proposition}
\begin{proof}
This result is a special case of~\cite[Proposition 5]{javanmard2013state}, and the reduction in explained in detail in~\cite[Appendix B]{celentano2020estimation}.    
\end{proof}

\subsection{Stage I: Finding the root of the tree}
The first stage of the algorithm is a simple instance of the above iteration where both $f_\ell$ depend only on the last iterates: 
$g_{\ell}(u^0,\cdots,u^\ell) = u^\ell$ and $f_{\ell}(v^0,\cdots,v^\ell) = f(v^\ell)$.
The iteration becomes
\begin{align}\label{eq:rs_amp}
\begin{split}
\u^{\ell+1} &= \bA^\top f(\v^\ell)  - \mathrm{b}_{\ell} \u^{\ell},\\  
\v^{\ell} &= \bA \,\, \u^\ell  -  f(\v^{\ell-1}), 
\end{split}
~~~~\ell \ge 0.
\end{align}    
The map $f$ and the coefficient $ \mathrm{b}_{\ell}$ are defined as 
\begin{equation}\label{eq:nonlin_rs}
f(x) = \lambda(\lbq) \partial_{x}\Phi_{\gamma_*}(\lbq, x), \qquad \mbox{and} \qquad
 \mathrm{b}_{\ell} =  \frac{1}{N} \sum_{k=1}^M  f'(v_k^\ell).
 \end{equation}
We initialize the iteration with $\u^0 =\sqrt{\frac{\lbq}{N}}\one$ and $\v^{0}=\bA\u^{0}$.  
Iteration~\eqref{eq:rs_amp} is the perceptron analogue of Bolthausen's iterative construction of solutions to the $\TAP$ equation for the $\SK$ model~\cite{bolthausen2014iterative}. We refer to it as \emph{replica-symmetric} approximate message passing, or $\RSAMP$.
 
\begin{remark}
We observe that if $f(0)=0$, then the pair $({\bm 0},{\bm 0})$ is an obvious fixed point of the above iteration. While this is the case for the Sherrington--Kirkpatrick model and mixed $p$-spin models with no external field (and $\lbq=0$), hence the absence of this first stage in the algorithms of~\cite{mon18} and~\cite{ams20}, this is not true for the one-sided perceptron problem. It is in fact easy to show that $\partial_x \Phi_{\gamma_*}(t,x)>0$ for all $(t,x) \in [0,1)\times \R$; see Section~\ref{sec:parisi_pde}, which implies $\lbq>0$ by virtue of the identity Eq.~\eqref{eq:stationary_q0}. 
\end{remark}
 In the remaining of this section we record for future use the asymptotic covariance structure of the iterates of $\RSAMP$. Let us first note that Eq.~\eqref{eq:stationary_q0} translates into the fixed-point equation
 \begin{equation}\label{eq:fixed_q0}
 \alpha\E\big[f(\sqrt{\lbq}Z)^2\big] = \lbq,
 \end{equation}
 where $Z \sim N(0,1)$.
Next, we define the sequence
\begin{equation}\label{eq:a_k}
a_k=\psi(a_{k-1}),~~~~~~~~ a_0 = 0,
\end{equation}
where
\[\psi(t) :=\alpha \E\big[f\big(\sqrt{t}Z+\sqrt{\lbq-t}Z'\big) f\big(\sqrt{t}Z+\sqrt{\lbq-t}Z''\big)\big],\]  
and $Z,Z',Z''$ are i.i.d.\ standard Gaussians.  
\begin{lemma}\label{lem:se_rs}
Let $(U^j)_{j=0}^\ell$ and $(V^j)_{j=0}^\ell$ be the random vectors toward which $(\u^1,\cdots, \u^{\ell})$ and $(\v^1,\cdots,\v^\ell)$ converge respectively, in the sense of Proposition~\ref{prop:state_evolution}. Then the following holds.
\begin{itemize}
\item For all $\ell \ge 0$,
\[\E\big[(U^\ell)^2\big] = \E\big[(V^\ell)^2\big] = \lbq.\]
\item For all $ 0 \le k < \ell$,
\[\E\big[U^\ell U^{k}\big] = \E\big[V^{\ell} V^{k}\big] = a_{k}.\]
\end{itemize}
\end{lemma}
\begin{proof}
From Proposition~\ref{prop:state_evolution} we see that $\E\big[(U^\ell)^2\big] = \E\big[(V^\ell)^2\big] = \alpha \E[f(V^{\ell-1})^2]$. Using the fixed point equation~\eqref{eq:fixed_q0} together with the initial condition $\E\big[(U^0)^2\big] = \lbq$, we obtain the first assertion by induction. Similarly, using Proposition~\ref{prop:state_evolution} and the first assertion, we have for $k < \ell$
\[ \E\big[V^{\ell} V^{k}\big] = \psi\big(\E\big[V^{\ell-1} V^{k-1}\big]\big),\]   
and $\E[V^{\ell}V^0] = 0$. Whence by induction, 
\[\E\big[V^{\ell} V^{k}\big] = \psi^{(k)}\big(\E[V^{\ell-k}V^0] \big) = a_{k}.\]
\end{proof}
\begin{lemma}\label{lem:psi}
The function $\psi$ is strictly increasing and strictly convex on $[0,\lbq]$. Moreover $\psi(\lbq)=\lbq$ and $\psi'(\lbq)=1.$ Finally $\psi(t)>t$ for all $t<\lbq$. 
\end{lemma}
\begin{proof}
Using Gaussian integration by parts, as in~\cite[Lemma 2.2]{bolthausen2014iterative}, we have
\begin{align*}
\psi'(t)=\alpha \E\big[f'\big(\sqrt{t}Z+\sqrt{\lbq-t}Z'\big) f'\big(\sqrt{t}Z+\sqrt{\lbq-t}Z''\big)\big],\\
\psi''(t)=\alpha \E\big[f''\big(\sqrt{t}Z+\sqrt{\lbq-t}Z'\big) f''\big(\sqrt{t}Z+\sqrt{\lbq-t}Z''\big)\big],
\end{align*}
and these expressions are seen to be positive by first fixing $Z$ and integrating in $Z',Z''$. The values of $\psi(\lbq),\psi'(\lbq)$ follow from Eq.~\eqref{eq:stationary}. The last claim follows by convexity of $\psi$ and the value $\psi'(\lbq)=1$ just established.
\end{proof}

\begin{corollary}\label{cor:a_k}
$\lim_{k \to \infty} a_k = \lbq$.
\end{corollary}
\begin{proof}
This follows easily from Lemma~\ref{lem:psi}. Indeed as $\psi(t )\in (t,\lbq)$ for $t<\lbq$ it follows that the sequence $(a_k)_{k\ge 0}$ is strictly increasing and converges to a limit $L\leq \lbq$. We have by continuity $L=\psi(L)$ and therefore $L=\lbq$, proving $\lim_{k\to\infty} a_k=\lbq$ as claimed. 
\end{proof}

\subsection{Stage II: Propagating down the tree}
After $\underline{\ell}$ iterations of $\RSAMP$, Eq.~\eqref{eq:rs_amp}, we move to an incremental stage where we employ the general iteration~\eqref{eq:general_amp} with $g_{\ell}(u^0,\cdots,u^\ell) = u^\ell$ and $f_{\ell}$ defined as follows.  
Let $a,b : [0,1] \times \R \to \R$ be two functions defined by 
\begin{equation}\label{eq:ab}
a(t,x) := \lambda(t) \partial_x^2\Phi_{\gamma_*}(t,x), \quad \mbox{and} \quad 
b(t,x) :=  \gamma_*(t) \partial_x\Phi_{\gamma_*}(t,x).
\end{equation}
We let
\begin{equation}\label{eq:delta}
\delta:=\Big(\frac{\lbq^2}{a_{\underline{\ell}}^2} -1\Big)\lbq.
\end{equation}
We note that $\delta \to 0$ as $\underline{\ell} \to \infty$ by Corollary~\ref{cor:a_k}.
Further, for $\ell \ge \underline{\ell}$, we let
\begin{equation}\label{eq:q_ell}
q_\ell := \lbq + (\ell - \underline{\ell})\delta.
\end{equation}
For $k\in [M]$, and given $v_k^0,\cdots,v_k^\ell$ with $\ell \ge \underline{\ell}+1$, we consider the finite difference equation
 \begin{align}\label{eq:discrete_cavity_field}
 \begin{split}
x_k^{j+1} - x_k^{j} &= b(q_j , x_k^{j}) \delta + (v_k^{j+1}-v_k^{j}), \quad \mbox{for} \quad \underline{\ell} \le j \le \ell-1, \\
x_k^{\underline{\ell}} &=v_k^{\underline{\ell}}.
\end{split}
\end{align}
We further let 
\begin{equation}\label{eq:def_a}
a^{\delta}_j (x) := \frac{a(q_j , x)}{\big(\alpha \E[a(q_j , X^{\delta}_j)^2]\big)^{1/2}},
\end{equation}
the random variables $X^{\delta}_{j}$ are defined further below, and
 \begin{align}\label{eq:discrete_cavity_magnetization}
 \begin{split}
  m_k^{\ell} &= m_k^{\underline{\ell}} + \sum_{j=\underline{\ell}}^{\ell-1} a^{\delta}_j (x_k^{j})(v_k^{j+1}-v_k^{j}),\quad \mbox{for} \quad \underline{\ell} +1 \le \ell,\\
m_k^{\underline{\ell}} &= (1+\eps_0)\lambda(\lbq) \partial_x \Phi_{\gamma_*}(\lbq, v^{\underline{\ell}}_k),
 \end{split}
 \end{align}
 where \[\eps_0 = \frac{\lbq}{a_{\underline{\ell}}} - 1.\] 
 (The value of $\eps_0$ is chosen to satisfy Eq.~\eqref{eq:orthog} below.)
We define the function $f_{\ell}$ as the one mapping $v_k^0,\cdots,v_k^\ell$ to $m_k^{\ell}$:
 \[f_{\ell} : (v_k^0,\cdots,v_k^\ell) \longmapsto m_k^{\ell} ~~~\mbox{as per Eq.~\eqref{eq:discrete_cavity_field} and Eq.~\eqref{eq:discrete_cavity_magnetization}}.\]  
Observe that $f_{\ell}$ depends only on $v_k^{\underline{\ell}},\cdots,v_k^{\ell}$.

A similar iteration was employed in the context of optimizing the $\SK$ hamiltonian in \cite{mon18} and its mixed $p$-spin generalizations in \cite{ams20}.  Following the terminology used in these papers, we refer to this second stage as \emph{incremental} approximate message passing, or $\IAMP$. 
It was shown in these papers that---in the context of $\SK$ and the $p$-spin model---$\IAMP$ has a scaling limit as $\delta \to 0$, $\ell \to \infty$ with $\ell \delta$ fixed, where all the involved iterates have a continuous time counterpart. 
We show that a similar result holds in our setting. 

We will see in Section~\ref{sec:parisi_pde} (in virtue of Eq.~\eqref{eq:limit_phi}) that the functions $a$ and $b$ are Lipschitz continuous in $x$, therefore so is $f_{\ell}$ for each $\ell$. We let $(U^{\delta}_j)_{j=\underline{\ell}}^{\ell}$ and $(V^{\delta}_j)_{j=\underline{\ell}}^{\ell}$ be the limits of the iterates $(\u^{\underline{\ell}},\cdots,\u^\ell)$ and $(\v^{\underline{\ell}},\cdots,\v^\ell)$ respectively, in the sense of Proposition~\ref{prop:state_evolution}. We also define $(X^{\delta}_j)_{j=\underline{\ell}}^{\ell}$ and $(M^{\delta}_j)_{j=\underline{\ell}}^{\ell}$ via the formulas~\eqref{eq:discrete_cavity_field} and~\eqref{eq:discrete_cavity_magnetization} by replacing the occurrences of $v_{k}^j$ with $V^{\delta}_j$.  
\begin{lemma}\label{lem:increments}
For $\ell \ge \underline{\ell}$, the centered Gaussian vectors $(U^{\delta}_j)_{j=\underline{\ell}}^{\ell+1}$ and $(V^{\delta}_j)_{j=\underline{\ell}}^{\ell+1}$ are such that
\begin{align*}
\E\big[\big(U^{\delta}_{\ell+1}-U^{\delta}_{\ell}\big)U^{\delta}_{j}\big] &= \E\big[\big(V^{\delta}_{\ell+1}-V^{\delta}_{\ell}\big)V^{\delta}_{j}\big] = 0 ~~~~ \mbox{for all } \underline{\ell}+1 \le j\le \ell,\\
\E\big[\big(U^{\delta}_{j+1}-U^{\delta}_{j}\big)^2\big] &= \E\big[\big(V^{\delta}_{j+1}-V^{\delta}_{j}\big)^2\big] = \delta ~~~~ \mbox{for all } \underline{\ell} \le j\le \ell.
\end{align*}
Consequently, $(M^{\delta}_j)_{j=\underline{\ell}+1}^{\ell+1}$ is a martingale with respect to $(V^{\delta}_j)_{j=\underline{\ell}+1}^{\ell+1}$. 
\end{lemma}
\begin{proof}
It can be seen from Eq.~\eqref{eq:cov_amp}, and $g_{\ell}(u^0,\cdots,u^\ell)=u^{\ell}$, that the vectors $(U^{\delta}_j)_{j=\underline{\ell}}^{\ell+1}$ and $(V^{\delta}_j)_{j=\underline{\ell}}^{\ell+1}$ have the same covariance structure, so we will only consider $(V^{\delta}_j)_{j=\underline{\ell}}^{\ell}$.  We proceed by induction. For the base case, using Lemma~\ref{lem:se_rs} we have
\begin{align}
\E\big[\big(V^{\delta}_{\underline{\ell}+1}-V^{\delta}_{\underline{\ell}}\big)V^{\delta}_{\underline{\ell}}\big] &= \alpha \E\big[M^{\delta}_{\underline{\ell}} f\big(V^{\underline{\ell}-1}\big)\big] - \lbq\nonumber\\
&= \alpha (1+\eps_0)  \E\big[ f\big(V^{\underline{\ell}}\big) f\big(V^{\underline{\ell}-1}\big)\big] - \lbq \nonumber\\
&=  (1+\eps_0)a_{\underline{\ell}} - \lbq = 0.\label{eq:orthog}
\end{align}

From the above we see that $\E\big[V^{\delta}_{\underline{\ell}+1} | V^{\delta}_{\underline{\ell}}\big] = V^{\delta}_{\underline{\ell}}$. 
Therefore, 
\begin{align*}
\E\big[\big(V^{\delta}_{\underline{\ell}+2}-V^{\delta}_{\underline{\ell}+1}\big)V^{\delta}_{\underline{\ell}+1}\big] &= \alpha \E\big[\big(M^{\delta}_{\underline{\ell}+1} -M^{\delta}_{\underline{\ell}}\big) M^{\delta}_{\underline{\ell}}\big]\\
&= \alpha \E\big[ a^{\delta}_{\underline{\ell}} (X^{\delta}_{\underline{\ell}}) \big(V^{\delta}_{\underline{\ell}+1} - V^{\delta}_{\underline{\ell}}\big)M^{\delta}_{\underline{\ell}}\big] \\
&=   \alpha (1+\eps_0)\E\big[ a^{\delta}_{\underline{\ell}} (V^{\delta}_{\underline{\ell}}) \big(V^{\delta}_{\underline{\ell}+1} - V^{\delta}_{\underline{\ell}}\big) f(V^{\delta}_{\underline{\ell}})\big] \\
&=0.
\end{align*}
Furthermore, using the above result and Lemma~\ref{lem:se_rs}, we have
\begin{align*}
\E\big[\big(V^{\delta}_{\underline{\ell}+1}-V^{\delta}_{\underline{\ell}}\big)^2\big] &= \E\big[\big(V^{\delta}_{\underline{\ell}+1}\big)^2\big] -\E\big[\big(V^{\delta}_{\underline{\ell}}\big)^2\big]\\
&= \alpha \E\big[\big(M^{\delta}_{\underline{\ell}}\big)^2\big] - \lbq\\
&= (1+\eps_0)^2\lbq -\lbq = \Big(\frac{\lbq^2}{a_{\underline{\ell}}^2} -1\Big)\lbq = \delta.
\end{align*}
 Next assume the above statements holds true up to $\ell-1$. Let $\underline{\ell}+1\le j \le \ell$. Using the definition of covariances in Eq.~\eqref{eq:cov_amp},  
\begin{align*}
\E\big[\big(V^{\delta}_{\ell+1}-V^{\delta}_{\ell}\big) V^{\delta}_{j}\big] &= \alpha \E\big[M^{\delta}_{\ell} M^{\delta}_{j-1}\big] - \alpha \E\big[M^{\delta}_{\ell-1} M^{\delta}_{j-1}\big]\\ 
&=\alpha \E\big[\big(M^{\delta}_{j-1}\big)^2\big] - \alpha \E\big[\big(M^{\delta}_{j-1}\big)^2\big] = 0,
\end{align*}
where the second line follows from the induction hypothesis. 
As for the variance of the last increment, we have
\begin{align*}
\E\big[\big(V^{\delta}_{\ell+1}-V^{\delta}_{\ell}\big)^2\big] &= \alpha \E\big[\big(M^{\delta}_{\ell}-M^{\delta}_{\ell-1}\big)^2\big] \\
&= \alpha \E\big[a_{\ell-1}^{\delta}(X^{\delta}_{\ell-1})^2 \big(V^{\delta}_{\ell}-V^{\delta}_{\ell-1}\big)^2\big] \\
&=  \alpha \E\big[a^{\delta}_{\ell-1}(X^{\delta}_{\ell-1})^2\big] \delta =\delta.
\end{align*}
In the above, the third equality follows since $X^{\delta}_{\ell-1}$ and $V^{\delta}_{\ell} - V^{\delta}_{\ell-1}$ are independent, and the last equality can be seen from Eq.~\eqref{eq:def_a}.
\end{proof}
Let $(B_t)_{t \in [0,1]}$ be a standard  Brownian motion, defined on the same probability space as all the previously defined random variables, such that
 \begin{equation}\label{eq:coupling}
B_{q_j} = V^{\delta}_{j}~~~\mbox{ for all }~ \underline{\ell}+1 \le j \le \ell.
 \end{equation}
(Recall that $q_j = \lbq + (j-\underline{\ell})\delta$.)
 This coupling is consistent with the marginals of $(B_t)_{t \in [0,1]}$ and $(V^{\delta}_j)_{j=\underline{\ell}}^{\ell}$ as characterized in Lemma~\ref{lem:increments}.  
Let $(\mF_t)_{t \in [0,1]}$ be its natural filtration. Consider the SDE~\eqref{eq:sde}:  
\[\rmd X_t  = \gamma_*(t) \partial_x \Phi_{\gamma_*}(t,X_t) \rmd t + \rmd B_t,\] 
with initial condition $X_0 = 0$. We observe that since $\gamma_*(t) = 0$ on $[0,\lbq]$, we have $X_{\lbq} = B_{\lbq}$.
We further define a new process $(M_t)_{t\in [\lbq,1)}$ by
\begin{align}\label{eq:M_t}
M_t &:=  \lambda(\lbq) \partial_x \Phi_{\gamma_*}(\lbq, X_{\lbq}) +  \int_{\lbq}^t \lambda(s)\partial_x^2\Phi_{\gamma_*}(s,X_s) \rmd B_s.
\end{align}
\begin{proposition}\label{prop:continuum}
Assume $ \ubq(\gamma_*)<1$. There exists $\delta_0>0$ and a constant $C>0$ such that for every $\underline{\ell}$ such that $\delta \le \delta_0$ and every $\ell \ge \underline{\ell}$ such that $q_{\ell} \le \ubq(\gamma_*)$, we have  
\begin{equation*}
\max_{ \underline{\ell} \le j \le \ell} \E\big[\big(X^{\delta}_{j}-X_{q_j}\big)^2\big] \le C\delta,
~~~~~\mbox{and}~~~~~
\max_{ \underline{\ell} \le j \le \ell} \E\big[\big(M^{\delta}_{j}-M_{q_j}\big)^2\big] \le C\delta.
\end{equation*} 
\end{proposition}
The proof of this proposition uses standard SDE approximation techniques and will be briefly delayed to the end of this section.     
In the next proposition, we characterize the asymptotic behavior of the coordinates of the vector $\bA \u^{\ell}$ in terms of the law of the process $(X_t)_{t\in [0,1]}$.
\begin{proposition}\label{prop:se_combined}
 Let $q \in [\lbq,\ubq]$. Let $(\u^\ell, \v^{\ell}) \in \R^N \times \R^{M}$, with $\ell = \underline{\ell}+\lfloor (q - \lbq)/\delta\rfloor$, be the output of the above two-stage algorithm with $\underline{\ell}$ iterations of $\RSAMP$ followed by $\lfloor (q-\lbq)/\delta\rfloor$ iterations of $\IAMP$ with $\delta$ defined in Eq.~\eqref{eq:delta}.  
Let $\psi$ be a pseudo-Lipschitz function. Then we have   
\begin{align} 
\lim_{\underline{\ell} \to \infty} \mathop{\plim_{N \to \infty}}_{M/N \to \alpha} \big\langle \psi\big( \bA \u^{\ell}\big) \big\rangle_M &= \E \psi\big(\E[X_1 | \mF_{q}]\big),\label{eq:psi_lim}\\
\mbox{and}~~~~~
\lim_{\underline{\ell} \to \infty} \mathop{\plim_{N \to \infty}}_{M/N \to \alpha} \frac{1}{N} \big\| \u^{\ell} \big\|_2^2 &= q.\label{eq:two_norm}
\end{align}
 \end{proposition}
\begin{proof}
We start with Eq.~\eqref{eq:two_norm}.  Using Proposition~\ref{prop:state_evolution} we obtain
\[ \mathop{\plim_{N \to \infty}}_{M/N \to \alpha} \frac{1}{N} \big\| \u^{\ell} \big\|_2^2 = \E\big[(U^{\delta}_{\ell})^2\big].\]
 Further,  using Lemma~\ref{lem:increments}  and Lemma~\ref{lem:se_rs} we have
 \[\E\big[(U^{\delta}_{\ell})^2\big] = \E\big[(U^{\delta}_{\underline{\ell}})^2\big]  + (\ell-\underline{\ell})\delta  = q_{\ell}.\] 
 We obtain the desired result since  $q_{\ell} \to q$ as $\underline{\ell} \to \infty$.
Now we turn our attention to Eq.~\eqref{eq:psi_lim}.
Rearranging the second step of the $\AMP$ iteration, Eq.~\eqref{eq:general_amp}, 
\begin{align*}
\bA \u^\ell &=  \v^{\ell} + \sum_{s=0}^{\ell} \mathrm{d}_{\ell, s} f_{s-1}(\v^0,\cdots,\v^{s-1})\\
&= \v^{\ell} + \sum_{s=0}^{\ell} \mathrm{d}_{\ell, s} \, \m^{s-1}\\
&= \v^{\ell} + \m^{\ell-1}.
\end{align*}
The last equality follows from $\mathrm{d}_{\ell, s} = \delta_{\ell ,s}$ due to the choice $g_{\ell}(u^0,\cdots,u^\ell) = u^{\ell}$.
Hence, using Proposition~\ref{prop:state_evolution}, we have 
\[\mathop{\plim_{N \to \infty}}_{M/N \to \alpha} \big\langle \psi\big( \bA \u^{\ell}\big) \big\rangle_M = \E \psi\big(V^{\delta}_\ell +  M^{\delta}_{\ell-1}\big).\] 
Now, by Proposition~\ref{prop:continuum} and the fact that $\psi$ is pseudo-Lipschitz, there exists a constant $C$ such if $q_{\ell}\le \ubq$ then 
\[\Big|\E \psi\big(V^{\delta}_\ell +  M^{\delta}_{\ell-1}\big) - \E \psi\big(B_{q_{\ell}} +  M_{q_{\ell-1}}\big)\Big| \le C \sqrt{\delta}.\] 
Now we claim that for all $ \lbq \le q\le1$, 
\[\E[X_1 | \mF_{q}] = B_{q} + M_{q}.\]
Assuming this, the claim Eq.~\eqref{eq:psi_lim} follows by the continuity of $B_t$ and $M_t$, and $q_{\ell} \to q$ as $\underline{\ell} \to \infty$. Now we prove the above identity. Using integration by parts together with the identity $\rmd \partial_x\Phi_{\gamma_*}(t,X_t) = \partial_x^2\Phi_{\gamma_*}(t,X_t) \rmd B_t$, we have
\begin{align*}
M_q &=  \lambda(\lbq) \partial_x \Phi_{\gamma_*}(\lbq, X_{\lbq}) + \int_{\lbq}^q \lambda(s)\partial_x^2\Phi_{\gamma_*}(s,X_s) \rmd B_s\\
&=  \lambda(\lbq) \partial_x \Phi_{\gamma_*}(\lbq, X_{\lbq}) + \Big[\lambda(s)\partial_x\Phi_{\gamma_*}(s,X_s)\Big]_{s=\lbq}^{s=q} + \int_{\lbq}^q \gamma_*(s)\partial_x\Phi_{\gamma_*}(s,X_s) \rmd s\\
&= \lambda(q) \partial_x\Phi_{\gamma_*}(q,X_q) + \int_{\lbq}^q \gamma_*(s)\partial_x\Phi_{\gamma_*}(s,X_s) \rmd s.
\end{align*}
On the other hand, if we integrate the SDE obeyed by $X_t$, we obtain
\begin{align*}
X_q  &= X_{\lbq} + \int_{\lbq}^q \gamma_*(s) \partial_x\Phi_{\gamma_*}(s,X_s) \rmd s + B_q - B_{\lbq}\\
&= B_q + \int_{\lbq}^q \gamma_*(s) \partial_x\Phi_{\gamma_*}(s,X_s) \rmd s,
\end{align*}
since $X_{\lbq} = B_{\lbq}$ by construction. Therefore 
\[B_q+M_q = X_q + \lambda(q) \partial_x\Phi_{\gamma_*}(q,X_q).\]
Furthermore, we have
\begin{align}\label{eq:cond_X}
\E[X_1 | \mF_{q}] &= X_{q} + \E \Big[\int_q^1 \gamma_*(t) \partial_x\Phi_{\gamma_*}(t,X_t) \rmd t \Big| \mF_q \Big]\nonumber\\
&= X_{q} + \Big(\int_q^1 \gamma_*(t) \rmd t \Big)\cdot \partial_x\Phi_{\gamma_*}(q,X_q)\nonumber\\
&= X_{q} + \lambda(q) \partial_x\Phi_{\gamma_*}(q,X_q),
\end{align}
where we used the martingale property of $t \mapsto \partial_x\Phi_{\gamma_*}(t,X_{t})$ to obtain the second equality.
\end{proof}

Next, the above conditional expectation has the desired property of being larger then $\kappa$:
\begin{lemma}\label{lem:ge_kappa}
For $(\kappa,\alpha) \in \Lambda$ and $q \in [\ubq,1]$ we have $\E[X_1 | \mF_{q}] \ge \kappa$ almost surely.
\end{lemma}
\begin{proof}
 By the Cole-Hopf transform we have $\Phi_{\gamma_*}(q,x) = \log \mN\left(\frac{\kappa - x}{\sqrt{1-q}}\right)$, whence 
\[\partial_x\Phi_{\gamma_*}(q,x) = \frac{1}{\sqrt{1-q}} \mA\left(\frac{\kappa - x}{\sqrt{1-q}}\right)\] for all $q \in [\ubq,1]$, where 
$\mA(x) := \frac{e^{-x^2/2}}{\int_x^{\infty} e^{-t^2/2}\rmd t}$ in the inverse Mill's ratio.
Therefore, using Eq.~\eqref{eq:cond_X}, we have
\[\E[X_1 | \mF_{q}] = X_{q} + \sqrt{1-q} \mA\left(\frac{\kappa - X_{q}}{\sqrt{1-q}}\right), \]
and from the lower bound $\mA(x) \ge x \vee 0$, we obtain
\[\E[X_1 | \mF_{q}]  \ge X_{q} + (\kappa- X_{q})_+ = \kappa \vee X_{q} \,. \]  
\end{proof}
Now we have all the ingredients to prove our main result:
\begin{proof}[Proof of Theorem~\ref{thm:main}] 
If we apply Proposition~\ref{prop:se_combined} with $\psi(x) = \max(\kappa-x,0)^2 = (\kappa-x)_+^2$, and use Lemma~\ref{lem:ge_kappa}, the vector $\u = \u^{\underline{\ell} + \lfloor (\ubq-\lbq)/\delta\rfloor}$ with $\underline{\ell}$ large enough is an approximate $\ell_2$ solution to the perceptron inequalities: 
\begin{equation} \label{eq:per_aprrox_l2}
\lim_{\underline{\ell} \to \infty} \mathop{\plim_{N \to \infty}}_{M/N \to \alpha} \frac{1}{ \sqrt{M}} \Big\|\Big(\kappa \one - \frac{1}{\sqrt{N}}\bG\u \Big)_+\Big\|_2 = 0.
\end{equation}
(Recall that $\delta \to 0$ as $\underline{\ell}\to \infty$ as per Eq.~\eqref{eq:delta}.)
Furthermore, $\u_{\ubq} $ is approximately on a sphere of radius $\sqrt{\ubq N}$:
\begin{equation}\label{eq:norm_aprrox_l2}
\lim_{\underline{\ell} \to \infty} \mathop{\plim_{N \to \infty}}_{M/N \to \alpha} \frac{1}{N}\big\| \u  \big\|_2^2 = {\ubq} .
\end{equation}
\end{proof}

\begin{remark}
Although the above result is stated for $q \ge \ubq$, it is possible to track the `quality' of $\u_q=\u^{\underline{\ell} + \lfloor (q-\lbq)/\delta\rfloor}$ for any value of $q \in [\lbq, \ubq]$. 
Indeed, a simple extension of the proof of Lemma~\ref{lem:ge_kappa} allows to show
\[\E[X_1 | \mF_q] \ge \kappa - \Big(\int_q^{\ubq} \gamma_*(t)\rmd t \Big) \cdot \partial_{x}\Phi(q,X_q)~~~\mbox{a.s.\ for all}~~ q \in [\lbq,\ubq].\]
We observe that the margin at time $q$, $\kappa_q :=  \kappa - \big(\int_q^{\ubq} \gamma_*(t)\rmd t \big) \, \partial_{x}\Phi(q,X_q)$, is a sub-martingale, i.e., $\E[\kappa_q | \kappa_\tau] \ge \kappa_\tau$ for $q \ge \tau$. 
Thus we have the following geometric picture:  $\RSAMP$ outputs an initialization on the sphere of radius $\sqrt{\lbq N}$. This is then incrementally updated by orthogonal additions which simultaneously grows the solution to larger and larger radii and improves its expected margin. The target margin $\kappa$ is reached at the right-end of the support of $\gamma_*$. 
\end{remark}

\begin{proof}[Proof of Proposition~\ref{prop:continuum}]
Fix $t_*<1$ and let $\ell  \ge \underline{\ell}+1 $ such that $q_{\ell}\le t_*$. Define $\Delta^{X}_{\ell} = X^{\delta}_{\ell} - X_{q_{\ell}}$. 
First, using Eq.~\eqref{eq:coupling}, the initial condition $X^{\delta}_{\underline{\ell}}=V^{\delta}_{\underline{\ell}}$ and $X_{\lbq} = B_{\lbq}$, we have
\begin{align*}
\Delta^{X}_{\underline{\ell}+1} &= X^{\delta}_{\underline{\ell}+1} - X_{q_{\underline{\ell}+1}} \\
&= X^{\delta}_{\underline{\ell}} - X_{q_{0}} +  \int_{q_{0}}^{q_{\underline{\ell}+1}} \big(b(q_{0}, X^{\delta}_{\underline{\ell}}) - b(t, X_{t}) \big) \rmd t + V^{\delta}_{\underline{\ell}+1}  - V^{\delta}_{\underline{\ell}} -B_{q_{\underline{\ell}+1}}  + B_{q_{0}}\\
&= \int_{q_{0}}^{q_{\underline{\ell}+1}} \big(b(\lbq, X^{\delta}_{\underline{\ell}}) - b(t, X_{t}) \big) \rmd t \\
&=\int_{q_{0}}^{q_{\underline{\ell}+1}} \big(b(\lbq, X^{\delta}_{\underline{\ell}}) - b(\lbq, X_{t}) \big) \rmd t + \int_{q_{0}}^{q_{\underline{\ell}+1}} \big(b(q_{0}, X_t) - b(t, X_{t}) \big) \rmd t.
\end{align*}
Next, for $ \underline{\ell}+2\le j \le \ell$, and using Eq.~\eqref{eq:coupling}, we have
\begin{align*}
\Delta^{X}_{j} - \Delta^{X}_{j-1} &= \int_{q_{j-1}}^{q_{j}} \big(b(q_{j-1}, X^{\delta}_{j-1}) - b(t, X_{t}) \big) \rmd t + V^{\delta}_{j}-V^{\delta}_{j-1} -  B_{q_{j}} + B_{q_{j-1}}\\
&=\int_{q_{j-1}}^{q_{j}} \big(b(q_{j-1}, X^{\delta}_{j-1}) - b(q_{j-1}, X_{t}) \big) \rmd t 
+ \int_{q_{j-1}}^{q_{j}} \big(b(q_{j-1}, X_{t}) - b(t, X_{t}) \big) \rmd t.
\end{align*}
Summing over $\underline{\ell}+2 \le j \le \ell$:
\begin{align}\label{eq:delta_X}
\big|\Delta^{X}_{\ell}\big| &\le \big|\Delta^{X}_{\underline{\ell}+1}\big| + \sum_{j=\underline{\ell}+2}^{\ell} \big|\Delta^{X}_{j}-\Delta^{X}_{j-1}\big| \nonumber\\
&\le  \sum_{j=\underline{\ell}+1}^{\ell} \int_{q_{j-1}}^{q_{j}} \big|b(q_{j-1}, X^{\delta}_{j-1}) - b(q_{j-1}, X_{t}) \big| \rmd t 
+ \int_{q_{j-1}}^{q_{j}} \big|b(q_{j-1}, X_{t}) - b(t, X_{t}) \big| \rmd t.
\end{align}
Since $x \mapsto b(t,x)$ is Lipschitz uniformly in $t \in [0,\ubq]$, the first term in the above display is bounded in absolute value by 
\[C\sum_{j=\underline{\ell}+1}^{\ell} \int_{q_{j-1}}^{q_{j}} \big|X^{\delta}_{j-1} - X_{t} \big| \rmd t, \] 
where $C$ depends only on $\ubq$. As for the second term, since $\partial_x\Phi_{\gamma_*}$ is non-negative and $\gamma_*$ is non-decreasing, we have 
\begin{align*}
\big|b(q_{j-1}, X_{t}) - b(t, X_{t})\big| &\le \gamma_*(t) \big| \partial_x\Phi_{\gamma_*}(q_{j-1}, X_{t}) - \partial_x\Phi_{\gamma_*}(t, X_{t})\big|\\
&\le \big| \partial_x\Phi_{\gamma_*}(q_{j-1}, X_{t}) - \partial_x\Phi_{\gamma_*}(t, X_{t})\big|.
\end{align*}
Plugging these bounds Eq.~\eqref{eq:delta_X}, then squaring and taking expectations, we obtain
\begin{align*}
\E\big[\big(\Delta^{X}_{\ell})^2\big] &\le C (\ell-\underline{\ell}) \delta \sum_{j=\underline{\ell}+1}^{\ell} \int_{q_{j-1}}^{q_{j}} \Big(\E\big[ \big(X^{\delta}_{j-1} - X_{t} \big)^2\big] \\
&\hspace{3cm} + \E\big[\big(\partial_x\Phi_{\gamma_*}(q_{j-1}, X_{t}) - \partial_x\Phi_{\gamma_*}(t, X_{t})\big)^2\big] \Big) \rmd t.
\end{align*}
Using item $(ii)$ of Lemma~\ref{lem:useful} we have $\E\big[ \big(X^{\delta}_{j-1} - X_{t} \big)^2\big]  \le 2\E\big[ \big(X^{\delta}_{j-1} - X_{q_{j-1}} \big)^2\big] +2(t-q_{j-1})$. Next, using item $(iii)$ of the same lemma, we have (since $(\ell-\underline{\ell}) \delta\le 1$)
\begin{align*}
\E\big[\big(\Delta^{X}_{\ell})^2\big] &\le C (\ell-\underline{\ell}) \delta^2 \sum_{j=\underline{\ell}+1}^{\ell} \E\big[\big(\Delta^{X}_{j-1})^2\big] + C(\ell-\underline{\ell})^2 \delta^3 + C(\ell-\underline{\ell})^2 \delta^4\\
&\le C \delta \sum_{j=\underline{\ell}}^{\ell-1} \E\big[\big(\Delta^{X}_{j})^2\big]  + C \delta.
\end{align*}
This implies $\E\big[\big(\Delta^{X}_{\ell})^2\big] \le C\delta$ for all $\ell \ge \underline{\ell}+1$. The case $\ell = \underline{\ell}$ must be dealt with separately: since $B_{q_{\underline{\ell}+1}}=V_{\underline{\ell}+1}$ we have
\begin{align*}
\E\big[\big(X^{\delta}_{\underline{\ell}} - X_{\lbq}\big)^2\big] &= \E\big[\big(V^{\delta}_{\underline{\ell}} - B_{\lbq}\big)^2\big]  \\
&\le  2\E\big[\big(V^{\delta}_{\underline{\ell}} - V^{\delta}_{\underline{\ell}+1}\big)^2\big] + 2\E\big[\big(B_{q_{\underline{\ell}+1}} - B_{\lbq}\big)^2\big] \\
&= 4\delta.
\end{align*}
where the last equality follows from Lemma~\ref{lem:increments}. 

Now we prove a similar bound for $M^{\delta}_{j} - M_{q_{j}}$. We have
 \begin{align*}
\E\big[\big(M^{\delta}_{\underline{\ell}} - M_{\lbq}\big)^2\big] &= \lambda(\lbq)^2\E\big[\big((1+\eps_0)\partial_x\Phi_{\gamma_*}(\lbq, X_{\underline{\ell}}) - \partial_x\Phi_{\gamma_*}(\lbq, X_{\lbq})\big)^2\big]  \\
&\le  2\eps_0^2\E\big[\big(\partial_x\Phi_{\gamma_*}(\lbq, X_{\underline{\ell}})\big)^2\big] 
+ 2\E\big[\big(\partial_x\Phi_{\gamma_*}(\lbq, X_{\underline{\ell}}) - \partial_x\Phi_{\gamma_*}(\lbq, X_{\lbq})\big)^2\big] \\
&= C\eps_0^2 + C\E\big[\big(X^{\delta}_{\underline{\ell}} - X_{\lbq}\big)^2\big]\\
&\le C\delta.
\end{align*}
The third line follows from item $(i)$ of Lemma~\ref{lem:useful} below, and the fact that $\partial_x\Phi_{\gamma_*}(\lbq, \cdot)$ is Lipschitz.
The last line follows from the previously established bound on $X^{\delta}_{\underline{\ell}} - X_{\lbq}$ and from the simple inequality $\eps_0^2 \le \delta/\lbq$. Now we let $\ell \ge \underline{\ell}+1$. Using the coupling in Eq.~\eqref{eq:coupling} we have
 \begin{align*}
\E\big[\big(M^{\delta}_{\ell} - M_{q_{\ell}}\big)^2\big] &\le 2\E\big[\big(M^{\delta}_{\underline{\ell}} - M_{\lbq}\big)^2\big] 
+   2\E\Big[\Big(\sum_{j=\underline{\ell}}^{\ell-1} \int_{q_{j}}^{q_{j+1}} \big(a^{\delta}_{j}(X^{\delta}_{j}) - a(t,X_t)\big) \rmd B_t\Big)^2\Big]
\\
&\le C\delta 
+2\sum_{j=\underline{\ell}}^{\ell-1} \int_{q_{j}}^{q_{j+1}} \E\big[ \big(a^{\delta}_{j}(X^{\delta}_{j}) - a(t,X_t)\big)^2\big] \rmd t.
\end{align*}
Recall that $a^{\delta}_{j}(x) = a(q_j,x) \big/ \big(\alpha \E[a(q_j,X^{\delta}_{j})^2]\big)^{1/2}$. At this point we argue that
\[\big|\alpha \E[a(q_j,X^{\delta}_{j})^2] - 1\big| \le C\sqrt{\delta}.\]
Indeed, since $a(q_j,\cdot)$ is Lipschitz, we have $\E[a(q_j,X^{\delta}_{j})^2] = \E[a(q_j,X_{q_j})^2] + \bigo(\sqrt{\delta})$. Furthermore, the stationarity property of $\gamma_*$ as a minimizer of $\Par$, Eq.~\eqref{eq:stationary}, implies by Lebesgue differentiation theorem that
\[\alpha \lambda(t)^2 \E\big[\big(\partial_x^2\Phi_{\gamma_*}(t, X_{t})\big)^2\big] =1 ~~~\mbox{for all}~ t \in [\lbq,\ubq].\]
This implies $\alpha\E[a(q_j,X_{q_j})^2]=1$ for all $\ell_\star \le j\le \ell$ such that $q_{\ell}\le \ubq$, which establishes the above claim. 
Since $a$ is bounded, we have
\begin{align*}
 \E\big[ \big(a^{\delta}_{j}(X^{\delta}_{j}) - a(t,X_t)\big)^2\big] &\le 2C\Big(\frac{1}{\big(\alpha \E[a(q_j,X^{\delta}_{j})^2]\big)^{1/2}}-1\Big)^2 
 +   2\E\big[ \big(a(q_j,X^{\delta}_{j}) - a(t,X_t)\big)^2\big]\\
 &\le C\delta +   2\E\big[ \big(a(q_j,X^{\delta}_{j}) - a(t,X_t)\big)^2\big],
\end{align*}
and we obtain
\[\E\big[\big(M^{\delta}_{\ell} - M_{q_{\ell}}\big)^2\big] \le C\delta + 4 \sum_{j=\underline{\ell}}^{\ell-1} \int_{q_{j}}^{q_{j+1}} \E\big[ \big(a(q_j,X^{\delta}_{j}) - a(t,X_t)\big)^2\big] \rmd t.\]
Using the same decomposition as in Eq.~\eqref{eq:delta_X}, and leveraging the Lipschitz property of $a(t,\cdot)$, together with the bound $(iv)$ of Lemma~\ref{lem:useful}, we obtain the desired result
\[\E\big[\big(M^{\delta}_{\ell} - M_{q_{\ell}}\big)^2\big] \le C\delta + C(\ell - \underline{\ell})\delta^2 + C(\ell - \underline{\ell})\delta^3 \le C\delta.\]  
\end{proof}

\subsection{Stage III: Rounding}
\label{sec:rounding}
In the last stage of the algorithm we can process the vector output by $\IAMP$ into a genuine solution to the perceptron inequalities which lies exactly on the sphere of radius $\sqrt{\ubq N}$. 

Given the vector $\u = \u^{\underline{\ell}+\lfloor (\ubq-\lbq)/\delta\rfloor}$ such that 
\begin{align*}
\Big\|\Big(\kappa \one - \frac{1}{\sqrt{N}}\bG\u\Big)_+\Big\|_2 &\le  \varepsilon_0 \sqrt{N}~~~ 
\mbox{and}~~~ \big|\|\u\|_2 - \sqrt{\ubq N}\big| \le \varepsilon_1 \sqrt{N}.
\end{align*}
we consider its projection on the set of hyperplanes~\eqref{eq:perceptron}:
\begin{equation} \label{eq:projection}
\bsigma^* = \argmin \left\{ \big\| \bsigma - \u\big\|_2 ~ :~ \bsigma \in \R^N, \, \, \bG \bsigma \ge \kappa \sqrt{N} \one \right\}.
\end{equation}
The above set of constraints has non-empty interior for $\kappa < 0$: $\bsigma = \mathbf{0}$ is a strictly feasible solution.
Furthermore, the optimization problem~\eqref{eq:projection} is convex and hence can be solved in polynomial time. For example the cutting plane method of~\cite{jiang2020improved} allows us to find a $\varepsilon_2$-approximate solution to the above problem in $ \bigo(N^3\log(N)^c\log(1/\varepsilon_2))$ time with high probability where $c>0$ is an absolute constant. By taking $\varepsilon_2$ sufficiently small and using the positive curvature of the sphere, this computes a point $\tilde{\bsigma}^*$ with $\|\bsigma^*-\tilde{\bsigma}^*\|_2\leq \varepsilon_2 \sqrt{N}$.
We let our final solution be $\hat{\bsigma} = \sqrt{\ubq N} \tilde{\bsigma}^* \big/ \|\tilde{\bsigma}^*\|_2$.
\begin{lemma}\label{lem:rounding}
Let $(\kappa,\alpha) \in \Lambda$, with $\alpha \neq 1$ and $\varepsilon >0$. There exists a constant $c_0 = c_0(\alpha)$ and $\ell_0\ge 1$ such that for all $\underline{\ell} \ge \ell_0$, 
we have $\bG \hat{\bsigma} \ge \frac{\kappa}{1-c_0\varepsilon} \sqrt{N} \one$ with probability converging to one as $N \to \infty$. 
\end{lemma}
The additional condition $\alpha \neq 1$ is not strong. In fact we can restrict our attention to $\alpha\geq 2$ in when $\kappa <0$ where full replica symmetry breaking is expected by the result of \cite{cover1965geometrical}. 
\begin{proof}
For a matrix $\bM$, $s_{\min}(\bM)$ denotes the smallest positive singular value of $\bM$. We have the following sequence of inequalities:
\begin{align*}
\frac{1}{2}  \big\| \bsigma^* - \u\big\|_2^2 &= \max_{\blambda \in \R^M_+ \,:\, \blambda \perp \ker(\bG^\top)} ~\blambda^\top \Big(\kappa\one - \frac{1}{\sqrt{N}} \bG\u\Big) - \frac{1}{2N} \big\| \bG^\top \blambda\big\|_2^2\\
&\le \max_{\blambda \in \R^M_+ \,:\, \blambda \perp \ker(\bG^\top)} ~\blambda^\top \Big(\kappa\one - \frac{1}{\sqrt{N}} \bG\u\Big)_+ - \frac{s_{\min}\big(\bG\big)^2}{2N} \big\|\blambda\big\|_2^2\\
&\le \max_{\blambda \in \R^M} ~\blambda^\top \Big(\kappa\one - \frac{1}{\sqrt{N}} \bG\u\Big)_+ - \frac{C}{2} \big\|\blambda\big\|_2^2\\
&= \frac{1}{2C}\Big\|\Big(\kappa\one - \frac{1}{\sqrt{N}} \bG\u\Big)_+\Big\|_2^2\\
&\le \frac{ \varepsilon_0^2 N}{2C} .
\end{align*} 
The first line follows from convex duality applied to the convex program~\eqref{eq:projection}. The second line follows by neglecting the non-positive part of the first term, and the fact that $\bG^\top$ is well conditioned on its range: classical facts from random matrix theory imply $C = \frac{1}{N}s_{\min}\big(\bG\big)^2  = \frac{(\sqrt{M} - \sqrt{N})^2}{N} - o_N(1)  = (\sqrt{\alpha}-1)^2 - o_N(1)>0$ for $\alpha \neq 1$.  The last line follows by assumption. 
Concluding, $\tilde{\bsigma}^*$ is an exact solution to the system of perceptron inequalities and is close to the sphere of radius $\sqrt{\ubq N}$ with high probability: 
\begin{align*}
\big| \|\tilde{\bsigma}^*\|_2 - \sqrt{\ubq N}| &\le \|\tilde{\bsigma}^*-\bsigma^*\|+\|\bsigma^* - \u\|_2 + \big| \|\u\|_2 - \sqrt{\ubq N}| \\
 &\le C(\varepsilon_2 + \varepsilon_0 + \varepsilon_1) \sqrt{N}= \varepsilon_3\sqrt{\ubq N}.
\end{align*}
Therefore we have with the same probability, 
\[\bG \hat{\bsigma} =  \sqrt{\ubq N} \frac{\bG\tilde{\bsigma}^*}{\|\tilde{\bsigma}^*\|_2} \ge  \frac{\bG\tilde{\bsigma}^*}{1\pm\varepsilon_3} \ge \frac{\kappa \sqrt{N}}{1-\varepsilon_3}.\]
(The sign $\pm$ in the above depends on whether the corresponding entry of the vector $\bG\tilde{\bsigma}^*$ is positive of negative.)
\end{proof}

\section{Properties of the variational principle}
\label{sec:var}
\subsection{A solution to the Parisi PDE}
\label{sec:parisi_pde}
 We consider the space of functions $\cuU$ endowed with the $L^1$ metric $d(\gamma_1,\gamma_2) = \|\gamma_1-\gamma_2\|_{L^1([0,1])} =  \int_0^1|\gamma_1(t)-\gamma_2(t)|\rmd t$. We study the $\PDE$
\begin{align}\label{eq:Parisi_PDE}
\begin{split}
\partial_t \Phi_{\gamma}  +\frac{1}{2}\Big(\gamma(t) (\partial_x \Phi_\gamma)^2 +  \partial^2_{x}\Phi_\gamma\Big) &=0~~\mbox{for all}~ (t,x) \in [0,\ubq)\times \R\\
\Phi_{\gamma}(\ubq,x) &= u_0(x),
\end{split}
\end{align}
for $\gamma \in \cuU$, $\ubq = \ubq(\gamma)$, and
\begin{equation}\label{eq:terminal_cond}
u_0(x) \equiv \log \mN\left(\frac{\kappa - x}{\sqrt{1-\ubq}}\right).
\end{equation}
We construct a classical solution to the above $\PDE$ via a limiting procedure, standard in literature on the Parisi formula, see e.g.,~\cite{talagrand2011mean2}, where $\gamma$ is  approximated by a piecewise constant function. We show uniform convergence of space derivatives of all orders with respect to this discretization. 
The main difference compared to settings previously considered in the literature is the terminal condition $u_0$ which, in contrast to $p$-spin models, is concave and diverges quadratically as $x \to -\infty$. As a result, the general theory of weak solutions, as used for instance in~\cite{jagannath2016dynamic,ams20} is more difficult to carry over to our setting. These technical difficulties must be dealt with by exploiting the specific features of $u_0$. 

As an illustrative example, we observe that since the coefficient associated to the non-linearity in the Parisi $\PDE$ (i.e., $\gamma$) is non-negative, solutions tend to gain positive curvature when evolved away from the terminal condition, so there is apriori no reason to expect that concavity is preserved at $t<\ubq$. We will show however that if $\gamma$ is non-decreasing, concavity of the solution holds at all times due to several fortunate cancellations.       

Consider the following collection of piecewise constant functions 
\begin{align}\label{eq:simple_functions}
\SF_+\equiv \biggl\{ \gamma=\sum_{i=0}^n m_i\indi_{[q_{i},q_{i+1})}:\;\; 
\begin{split} 
&0=\lbq<q_1<\cdots<q_{n+1}=1,\\ 
&0 = m_0 \le m_1 \le \cdots \le m_n=1, 
\end{split}
\, n \in \N  \biggr\}. 
\end{align}
 
 We record in the next proposition the main properties of $\Phi_{\gamma}$ together with some useful estimates when $\gamma \in \SF_+$.
 \begin{proposition}\label{prop:properties}
For $\gamma \in \SF_+$, the Parisi $\PDE$~\eqref{eq:Parisi_PDE} has a  classical solution $\Phi_{\gamma}$ on $[0,\ubq)\times \R$ which is $C^{\infty}([q_i,q_{i+1})\times \R)$ for all $0 \le i\le n-1$.
Moreover, there exists a constant $C = C(\kappa,\ubq)$, with $\ubq = \ubq(\gamma) = q_n$ such that for all $(t,x) \in [0,\ubq] \times\R$, the following bounds hold:  
\begin{align}
 -\frac{3}{2}\frac{(\kappa-x)^2}{1-\ubq} - C \le &\Phi_{\gamma}(t,x)\le 0,\label{eq:phi}\\
0 \vee \frac{\kappa-x}{\int_t^1 \gamma(t)\rmd t}\le &\partial_x\Phi_{\gamma}(t,x) \le C\left(1 + \frac{\kappa - x}{\sqrt{1-\ubq}} \vee 1 \right),\label{eq:dphi}\\
 - \frac{1}{1-\ubq} \le  &\partial_x^2\Phi_{\gamma}(t,x) \le 0.\label{eq:ddphi}
\end{align}
Finally, for all $k \ge 3$, there exists $C = C(\ubq,k)>0$ such that for all $(t,x) \in [0,\ubq] \times\R$,
\begin{equation}\label{eq:higher_dphi}
\big|\partial_x^k\Phi_{\gamma}(t,x)\big| \le C.
\end{equation}
\end{proposition}
Using the above estimates, we establish continuity of the Parisi solution and its derivatives as a function of $\gamma$. 
\begin{proposition}\label{prop:convergence_phi}
Let $\gamma_1, \gamma_2 \in \SF_+$ and $\ubq = \ubq(\gamma_1) \vee \ubq(\gamma_2)$. Let $K \subset \R$ be a bounded set and $k \ge 0$. There exists $C = C(\kappa,\ubq,K,k)>0$ such that
\[\sup_{t \in [0,\ubq]} \, \sup_{x \in K} \, \left| \partial_x^k\Phi_{\gamma_1}(t,x) - \partial_x^k\Phi_{\gamma_2}(t,x)\right| \le C  d(\gamma_1,\gamma_2).\]
\end{proposition}
\begin{proof}
The argument relies on SDE techniques, and follows an approach of Jagannath-Tobasco. We start with the case $k=0$. Since $ \Phi_{\gamma_1}$ and $ \Phi_{\gamma_2}$ are classical solutions to the Parisi $\PDE$~\eqref{eq:Parisi_PDE}, their difference $\omega = \Phi_{\gamma_1} - \Phi_{\gamma_2}$ satisfies the equation
\[\partial_t\omega + \frac{1}{2}\gamma_1(\partial_x\Phi_{\gamma_1})^2 - \frac{1}{2}\gamma_2(\partial_x\Phi_{\gamma_2})^2 +  \frac{1}{2}\partial^2_{x}\omega = 0,~~~\mbox{and}~~ \omega(\ubq,\cdot)=0.\]
After rearranging terms, the $\PDE$ becomes 
\begin{equation}\label{eq:pde_omega}
\partial_t\omega + v\partial_x \omega +  \frac{1}{2}\partial^2_{x}\omega + \frac{1}{2}(\gamma_1-\gamma_2)\big(\partial_x\Phi_{\gamma_2}\big)^2 = 0,
\end{equation}
with $v = \frac{1}{2}\gamma_1 (\partial_x \Phi_{\gamma_1}+\partial_x \Phi_{\gamma_2})$.
Therefore, the function $\omega$ has the Feynman-Kac representation
\[\omega(t,x) = \E_{t,x}\big[\omega(\ubq,Y_{\ubq})\big] + \frac{1}{2}\E_{t,x}\int_t^{\ubq} (\gamma_1(s)-\gamma_2(s))\big(\partial_x\Phi_{\gamma_2}(s,Y_s)\big)^2 \rmd s,\] 
where $(Y_t)$ solves the SDE
\begin{equation}\label{eq:sde_y}
\rmd Y_t = v(t,Y_t)\rmd t + \rmd B_t,
\end{equation}
$(B_t)$ is a standard Brownian motion, and $\E_{t,x}$ is the expectation conditional on $Y_t=x$. The bounds \eqref{eq:dphi} and \eqref{eq:ddphi} of Proposition~\ref{prop:properties} ensure that the above SDE has a unique strong solution for every initial condition $(t,x)$. 
Next, assume that for all $0 \le t \le s \le \ubq$, and all $x\in\R$,
\begin{equation}\label{eq:second_moment_Y}
\E_{t,x}\big[\big(\partial_x\Phi_{\gamma_2}(s,Y_s)\big)^2\big] \le C\Big(1+\frac{(\kappa-x)^2}{1-\ubq}\Big).
\end{equation}
Then by Fubini and the fact $\omega(\ubq,\cdot)=0$, we obtain
\[|\omega(t,x)| \le \frac{1}{2}C\Big(1+\frac{(\kappa-x)^2}{1-\ubq}\Big)\int_t^{\ubq}|\gamma_1(s)-\gamma_2(s)|\rmd s.\] 
This yields the desired bound $|\omega(t,x)| \le C d(\gamma_1,\gamma_2)$ uniformly over $t \le\ubq$ and $x$ in a bounded set. 

We now prove Eq.~\eqref{eq:second_moment_Y}.
We fix the pair $(t,x)$ and for $s \in [t,\ubq)$, we define
\[a(s) = \E_{t,x}\big[\big(\partial_x\Phi_{\gamma_1}(s,Y_s)\big)^2\big] \vee \E_{t,x}\big[\big(\partial_x\Phi_{\gamma_2}(s,Y_s)\big)^2\big].\]
We use the upper bound in Eq.~\eqref{eq:dphi} of Proposition~\ref{prop:properties} to obtain
\[a(s) \le C\Big(1+\frac{\E_{t,x}\big[(\kappa-Y_s)^2\big]}{1-\ubq}\Big).\]
Next, exploiting the SDE solved by $(Y_t)$, 
\begin{align*}
\E_{t,x}\big[(\kappa-Y_s)^2\big] &\le 3(\kappa-x)^2 + 3\E_{t,x} \Big(\int_t^s v(r,Y_r) \rmd r\Big)^2 + 3\E[(B_s-B_t)^2],\\
&\le 3(\kappa-x)^2 + 3(s-t)\E_{t,x} \int_t^s v(r,Y_r)^2 \rmd r + 3(s-t)\\
&\le 3(\kappa-x)^2 + 3\int_t^s a(r) \rmd r + 3,
\end{align*}
where we used the fact $\gamma_1(t )\le 1$ in the last inequality. So we obtain
\[a(s) \le C\Big(1+\frac{(\kappa-x)^2}{1-\ubq}+\int_t^s a(r) \rmd r\Big),\]
from which we conclude using Gronwall's lemma.

Next we consider the case $k\ge 1$ and proceed by induction.  We let $\phi_k = \partial_x^k \omega$. Taking $k$ derivatives of the $\PDE$~\eqref{eq:pde_omega},
\[\partial_t\phi_k + \partial_x^k\big(v\partial_x \omega\big) +  \frac{1}{2}\partial^{2}_{x}\phi_k + \frac{1}{2}(\gamma_1-\gamma_2) \partial_x^k\big[\big(\partial_x\Phi_{\gamma_2}\big)^2\big] = 0,\]
which we rewrite as
\[\partial_t\phi_k + v\partial_x \phi_k + k (\partial_x v) \phi_k +  \frac{1}{2}\partial^{2}_{x}\phi_k  =  - \sum_{j=2}^k {k\choose j} (\partial_x^j v) \phi_{k-j+1} - \frac{1}{2}(\gamma_1-\gamma_2) \partial_x^k\big[\big(\partial_x\Phi_{\gamma_2}\big)^2\big].\]
Applying the Feynman-Kac formula in conjunction with $\phi_k(\ubq, \cdot) = 0$, we obtain
\begin{align*}\label{eq:phi_k_FK}
 \phi_k(t,x) &=  \sum_{j=2}^k {k\choose j} \E_{t,x}\int_t^{\ubq} \beta(s) \partial_x^j v(s,Y_s) \phi_{k-j+1}(s,Y_s) \rmd s \\
&~~ + \frac{1}{2}\E_{t,x} \int_t^{\ubq} (\gamma_1(s)-\gamma_2(s)) \beta(s) \partial_x^k\big[\big(\partial_x\Phi_{\gamma_2}\big)^2\big](s,Y_s) \rmd s,\nonumber
\end{align*}
with $\beta(s) = \exp \big(k\int_t^s \partial_x v(r,Y_r) \rmd r\big)$. Recall that $v =  \frac{1}{2}\gamma_1 (\partial_x \Phi_{\gamma_1}+\partial_x \Phi_{\gamma_2})$ and $(Y_t)$ satisfies the SDE~\eqref{eq:sde_y}. From Proposition~\ref{prop:properties}, we see that $\beta(s)\le 1$; see Eq.~\eqref{eq:ddphi}, and $\|\partial_x^j v\|_{L^{\infty}([0,\ubq]\times \R)}<\infty$ for all $j \ge 2$; see Eq.~\eqref{eq:higher_dphi}. 

In order to conclude, it suffices to show that for all $k \ge 1$ and any bounded set $K \subset \R$, we have
\begin{equation}\label{eq:infinity_norm}
\sup_{s\in [0,\ubq]} \left\|\E_{t,x}\Big[ \partial_x^k\big[\big(\partial_x\Phi_{\gamma_2}\big)^2\big](s,Y_s)\Big]\right\|_{L^{\infty}([0,s]\times K)} <\infty. \end{equation}
Indeed, by induction we then obtain $\|\phi_k\|_{L^{\infty}([0,\ubq)\times K)} \le C d(\gamma_1,\gamma_2)$ where $C$ depends only on $k$ and the above $L^{\infty}$ norm.

Let $\psi_k = \partial_x^k\Phi_{\gamma}$. We have
\[\frac{1}{2} \E_{t,x}\Big[ \partial_x^k\big[\big(\partial_x\Phi_{\gamma_2}\big)^2\big](s,Y_s)\Big] = \frac{1}{2} \sum_{j=0}^k {k\choose j} \E_{t,x}\big[\psi_{j+1}(s,Y_s)\psi_{k-j+1}(s,Y_s)\big].\]
So the above left-hand side in bounded in absolute value by
\begin{align*}
\E_{t,x}\big[\big| \psi_1(s,Y_s) \psi_{k+1}(s,Y_s)\big|\big] + \frac{1}{2}  \sum_{j=1}^{k-1} {k\choose j} \|\psi_{j+1}\|_{L^{\infty}([0,\ubq]\times \R)} \cdot \|\psi_{k-j+1}\|_{L^{\infty}([0,\ubq]\times \R)}.
\end{align*}
Since $\Phi_{\gamma}$ has bounded second and higher derivatives (Proposition~\ref{prop:properties}), the second term in the above expression is finite. As for the first term, it is bounded by
\[ \|\psi_{k+1}\|_{L^{\infty}([0,\ubq]\times \R)} \cdot \E_{t,x}\big[\psi_1(s,Y_s)^2]^{1/2}  \le \|\psi_{k+1}\|_{L^{\infty}([0,\ubq]\times \R)} \cdot C (1+ |\kappa-x|),\]
in virtue of Eq.~\eqref{eq:second_moment_Y}. We have thus established the bound~\eqref{eq:infinity_norm}, and this concludes the proof of Proposition~\ref{prop:convergence_phi}. 
\end{proof}

As a consequence of Proposition~\ref{prop:convergence_phi}, given $\gamma \in \cuU$ and a sequence of functions $\gamma_n \in \SF_+$ such that $d(\gamma_n,\gamma) \to 0$ and $\ubq(\gamma_n) = \ubq(\gamma)$, the sequence of functions $\Phi_{\gamma_n}$ converges pointwise to a limit $\Phi_{\gamma}$ such that for all $t \in [0,\ubq]$,  $\Phi_{\gamma}(t,\cdot) \in C^{k}(\R)$. Furthermore, we have  
\begin{equation}\label{eq:limit_phi}
\forall \, j \ge 0,~ (t,x) \in [0,\ubq]\times \R, ~~~~\partial_x^j\Phi_{\gamma}(t,x) = \lim_{n \to \infty} \partial_x^j\Phi_{\gamma_n}(t,x),
\end{equation}
and the bounds \eqref{eq:phi}, \eqref{eq:dphi}, \eqref{eq:ddphi}, and \eqref{eq:higher_dphi} of Proposition~\ref{prop:properties} hold for $\Phi_{\gamma}$.
\begin{lemma}\label{lem:time_derivative}
 Let $\gamma \in \cuU$. For all $j \ge 0$ and all $(t,x) \in [0,\ubq(\gamma)]\times \R$, the partial derivatives $\partial_{t^{\pm}}\partial_x^{j}\Phi_{\gamma}$ exist. (Here, $\partial_{t^{\pm}}f$ are the left/right derivatives of the function $f$.) Furthermore, $\partial_{t^{+}}\partial_x^{j}\Phi_{\gamma}(t,x)=\partial_{t^{-}}\partial_x^{j}\Phi_{\gamma}(t,x)$ at continuity points of $\gamma$. 
\end{lemma}
\begin{proof}
 Since $\gamma$ is non-decreasing and c\`adl\`ag, we can decompose $[0,\ubq)$ into a disjoint union of countably many intervals $[a_i,b_i)$ such that $\gamma$ is continuous on $[a_i,b_i)$ and has a finite limit from the left at $b_i$. We consider a sequence $\gamma_n \in \SF_+$ converging uniformly to $\gamma$ on each interval $[a_i,b_i)$. 
 Exploiting the Parisi $\PDE$ and the uniform convergence of the partial derivatives $\partial_x^k\Phi_{\gamma_n}$ on $[0,\ubq) \times K$, where $K$ is a bounded set, we see that $\partial_t \partial_x^k\Phi_{\gamma_n}$ converges to a function $f_k$ uniformly on each $[a_i,b_i)$. Therefore, the limit $\partial_x^k\Phi_{\gamma}$ is differentiable with respect to time and $\partial_t \partial_x^k\Phi_{\gamma} = f_k$ on each $[a_i,b_i)$. Limit from the left at $b_i$ implies the existence of the left derivative at $b_i$.  
\end{proof}

\subsection{Analysis of the variational problem} 
\label{sec:analysis_var}
Having established the above regularity properties for $\Phi_{\gamma}$ with $\gamma \in \cuU$, we are now in a position to use stochastic calculus on $\Phi_{\gamma}$. Associated to the Parisi $\PDE$ is the SDE
\begin{equation}\label{eq:SDE2}
 \rmd X_t = \gamma(t) \partial_x\Phi_{\gamma}(t,X_t) \rmd t + \rmd B_t,~~~X_{t_0}=x.
 \end{equation}
From Lemma~\ref{lem:time_derivative}, and the bounds  \eqref{eq:dphi} and \eqref{eq:ddphi} of Proposition~\ref{prop:properties} , we see that the SDE~\eqref{eq:SDE2} has a unique strong solution for every initial data $(t_0,x) \in [0,\ubq(\gamma))\times \R$ when $\gamma \in \cuU$. We state some useful bounds which will be needed throughout.   
\begin{lemma}\label{lem:useful}
For all $t_*<1$, there exists a constant $C = C(\kappa,t_*)>0$ such that for all $s,t \in [0, t_*]$:
\begin{itemize}
\item[(i)] $\E\big[\big(\partial_x\Phi_{\gamma_*}(t, X_{t}) \big)^2\big] \le C$.
\item[(ii)]  $\E\big[\big(X_t - X_s\big)^2\big] \le C|t-s|$.
\item[(iii)] $\E\big[\big(\partial_x\Phi_{\gamma_*}(t, X_{t}) - \partial_x\Phi_{\gamma_*}(s, X_{t})\big)^2\big] \le C(t-s)^2$.
\item[(iv)] $\E\big[\big(\partial_x^2\Phi_{\gamma_*}(t, X_{t}) - \partial_x^2\Phi_{\gamma_*}(s, X_{t})\big)^2\big] \le C(t-s)^2$.
\end{itemize}
\end{lemma}
\begin{proof}
 Using the identity $\rmd \partial_x\Phi_{\gamma_*}(t,X_t) = \partial_x^2\Phi_{\gamma_*}(t,X_t) \rmd B_t$ and It\^o's isometry we have
 \[\E\big[\big(\partial_x\Phi_{\gamma_*}(t, X_{t}) \big)^2\big]  = \E\big[\big(\partial_x\Phi_{\gamma_*}(0, 0) \big)^2\big] + \int_0^t \E\big[\big(\partial_x^2\Phi_{\gamma_*}(s, X_{s}) \big)^2\big]  \rmd s.\] 
Item $(i)$ follows from the bound $|\partial_x^2\Phi_{\gamma_*}(t,x)| \le \frac{1}{1-t_*}$ for all $t \le t_*$ and $x\in \R$. Next, let $0\le s \le t\le 1$. We have
\begin{align*}
\E\big[\big(X_t - X_s\big)^2\big] &\le 2 \E\Big[\Big(\int_s^t \partial_x\Phi_{\gamma_*}(r, X_{r}) \rmd r\Big)^2\Big] + 2\E\big[\big(B_t - B_s\big)^2\big]\\
&\le 2(t-s)\int_s^t \E\big[\big(\partial_x\Phi_{\gamma_*}(r, X_{r})\big)^2\big] \rmd r + 2(t-s)\\
&\le C(t-s),
\end{align*} 
where we used item $(i)$ and the bound $s,t \le 1$. We now prove item $(iii)$:
\[\E\big[\big(\partial_x\Phi_{\gamma_*}(t, X_{t}) - \partial_x\Phi_{\gamma_*}(s, X_{t})\big)^2\big] \le (t-q)^2 \E\Big[\sup_{q \le s \le t} \big(\partial_t \partial_x\Phi_{\gamma_*}(s, X_{t})\big)^2\Big].\]
Exploiting the Parisi $\PDE$ we have $\partial_t \partial_x\Phi_{\gamma_*}(t,x) = - \gamma_*(t)  \partial_x^2\Phi_{\gamma_*}(t,x)  \partial_x\Phi_{\gamma_*}(t,x) - \frac{1}{2} \partial_x^3\Phi_{\gamma_*}(t,x)$. 
Since the second and third space derivative are uniformly bounded on $[0,t_*]$, it remains to bound $\E\big[\sup_{q \le s \le t} \big(\partial_x\Phi_{\gamma_*}(s, X_{t})\big)^2\big]$. To this end  we use the upper bound $\partial_x\Phi_{\gamma_*}(s, X_{t}) \le C \big(1+ \frac{\kappa-X_t}{\sqrt{1-t_*}}\vee 1\big)$ combined with the bound obtained from item $(ii)$: $\E[X_t^2] \le C$. The proof of $(iv)$ follows the exact same idea.  
\end{proof}

We can now deduce the optimality conditions for an optimizer $\gamma \in \cuU$ of the Parisi functional: 
\begin{proposition}\label{prop:optimality_gamma}
Assume $\gamma_* \in \cuU$ is such that $\Par(\gamma_*) = \inf_{\gamma \in \cuU} \Par(\gamma)$ and $\gamma$ is strictly increasing on $I = [\lbq(\gamma_*),\ubq(\gamma_*)]$. Then for all $t \in I$,
\begin{align*} 
\alpha \E\Big[\big(\partial_x \Phi_{\gamma_*}(t,X_t)\big)^2\Big] =  \int_0^t \frac{\rmd q}{\lambda(q)^2},\quad \mbox{and} \quad
\alpha \lambda(t)^2\E\Big[\big(\partial_x^2 \Phi_{\gamma_*}(t,X_{t})\big)^2\Big] = 1.
\end{align*}
\end{proposition}
\begin{proof}
The proof is almost identical to the proof of Proposition 6.8 in~\cite{ams20}.   Let $\gamma \in \cuU$ and $\delta : [0,1] \to \R$ be a bounded function vanishing outside the interval $I$. Let $\gamma^s = \gamma + s\delta$. We have
 \[\phi_{\gamma^s}(t,x) - \phi_{\gamma}(t,x) =  \frac{1}{2}\int_t^{1} (\gamma^s(r)-\gamma(r))\E_{t,x}\Big[\big(\partial_x\Phi_{\gamma^s}(r,Y^s_r)\big)^2\Big] \rmd s,\]
 where $(Y^s_r)_{r \in [0,1]}$ solves the SDE $ \rmd Y^s_t = v(t,Y_t^s) \rmd t + \rmd B_t$, $v(t,x) = \frac{1}{2}\big(\partial_x\Phi_{\gamma^s}(t,x)+\partial_x\Phi_{\gamma}(t,x)\big)$. We will show that 
 \begin{equation}\label{eq:YOs}
 \E\Big[\big(\partial_x\Phi_{\gamma^s}(r,Y^s_r)\big)^2\Big] = \E\Big[\big(\partial_x\Phi_{\gamma}(r,X_r)\big)^2\Big] + O(s),
 \end{equation}
 uniformly in $r \in [0,\ubq]$. Assuming for a moment that this is the case, we immediately obtain
 \[\frac{\rmd \phi_{\gamma_* + s\delta}(0,0)}{\rmd s} \Big|_{s=0} = \frac{1}{2}\int_0^{1} \delta(r) \E\Big[\big(\partial_x\Phi_{\gamma_*}(r,X_r)\big)^2\Big] \rmd s.\]
Furthermore, letting $\lambda^s(q) = \int_t^1 \gamma^s(r) \rmd r$, a simple computation yields 
\[\frac{\rmd}{\rmd s} \int_0^{\ubq} \frac{\rmd q}{ \lambda^s(q)} \Big|_{s=0} = -\int_0^{\ubq} \frac{\int_q^1\delta(r) \rmd r}{ \lambda(q)^2} \rmd q
= -\int_0^{1} \delta(r) \int_0^{r}  \frac{\rmd q}{ \lambda(q)^2} \rmd r.\] 
This yields the expression 
\begin{equation}\label{eq:differential}
\frac{\rmd \Par}{\rmd s} \big(\gamma_* + s\delta\big) \Big|_{s=0} = \frac{1}{2} \int_0^1 \delta(t) \left(\alpha\E\big[\big(\partial_x \Phi_{\gamma_*}(t,X_t)\big)^2\big] - \int_0^t \frac{\rmd q}{\lambda(q)^2}\right)\rmd t  .
\end{equation}
Since $\gamma_*$ is in the interior of $\cuU$ by assumption, then the optimality of $\gamma_*$ as a minimizer of $\Par$ implies that the above differential is zero for all $\delta$ as considered above. Therefore 
\[\alpha \E\Big[\big(\partial_x \Phi_{\gamma_*}(t,X_t)\big)^2\Big] =  \int_0^t \frac{\rmd q}{\lambda(q)^2}\] 
for all $t \in [\lbq,\ubq]$. Further, since we have $\frac{\rmd}{\rmd t} \E\big[\big(\partial_x \Phi_{\gamma_*}(t,X_t)\big)^2\big] = \E\big[\big(\partial_x^2 \Phi_{\gamma_*}(t,X_t)\big)^2\big]$, we obtain by differentiating,
\[\alpha \lambda(t)^2\E\Big[\big(\partial_x^2 \Phi_{\gamma_*}(t,X_{t})\big)^2\Big] = 1.\]  
It remains to show~\eqref{eq:YOs}:
\begin{align*} 
&\Big| \E\Big[\big(\partial_x\Phi_{\gamma^s}(r,Y^s_r)\big)^2\Big] - \E\Big[\big(\partial_x\Phi_{\gamma}(r,X_r)\big)^2\Big] \Big| \\
&\le \Big| \E\Big[\big(\partial_x\Phi_{\gamma^s}(r,Y^s_r)\big)^2\Big] -  \E\Big[\big(\partial_x\Phi_{\gamma}(r,Y^s_r)\big)^2\Big] \Big| 
+ \Big| \E\Big[\big(\partial_x\Phi_{\gamma}(r,Y^s_r)\big)^2\Big] - \E\Big[\big(\partial_x\Phi_{\gamma}(r,X_r)\big)^2\Big] \Big| \\
&\le C s \E\Big[\partial_x\Phi_{\gamma^s}(r,Y^s_r) + \partial_x\Phi_{\gamma}(r,Y^s_r)\Big] 
+ C  \E\Big[\big|Y^s_{r} - X_r\big| \cdot \big(\partial_x\Phi_{\gamma}(r,Y^s_r) + \partial_x\Phi_{\gamma}(r,X_r)\big)\Big],
\end{align*}
where we used the fact that $x\mapsto \partial_x\Phi_{\gamma}(r,x)$ is Lipschitz in $x$ uniformly in $r \le \ubq$. 
We can show that $ \E\big[\partial_x\Phi_{\gamma^s}(r,Y^s_r)^2\big] \le C $ uniformly in $r \le \ubq$, in a way similar to the proof of item $(i)$ in Lemma~\ref{lem:useful}. It remains to show that $\E\big[\big|Y^s_{r} - X_r\big|^2\big] \le C s^2$. We have for $t \le ubq$,
\begin{align*}
|Y^s_t -X_t | &\le \int_0^t \big|v(r,Y^s_r) - \gamma(r)\partial_x\Phi_{\gamma}(r,X_r) \big| \rmd r \\
&\le \frac{1}{2} \int_0^t \gamma^s(r)\big|\partial_x\Phi_{\gamma^s}(r,Y^s_r) - \partial_x\Phi_{\gamma}(r,Y^s_r) \big| \rmd r \\
&~~~+ \int_0^t \gamma^s(r)\big|\partial_x\Phi_{\gamma}(r,Y^s_r) - \partial_x\Phi_{\gamma}(r,X_r) \big| \rmd r\\
&~~~+ \int_0^t \big|\gamma^s(r)-\gamma(r)\big|\partial_x\Phi_{\gamma}(r,X_r) \rmd r\\
&\le C s + \int_0^t \gamma(t) |Y^s_r - X_r| \rmd r + s \int_0^t |\delta(r)|  \partial_x\Phi_{\gamma}(r,X_r) \rmd r.
\end{align*} 
Therefore,
\[\E\big[|Y^s_t -X_t |^2\big] \le Cs^2 + \int_0^t \E\big[|Y^s_r - X_r|^2\big] \rmd r,\]
and we conclude via Gronwall's inequality.
\end{proof}

\subsection{Proof of Proposition~\ref{prop:properties}}
For $\gamma \in \SF_+$, we can exhibit an explicit solution to the Parisi $\PDE$~\eqref{eq:Parisi_PDE} via the Cole-Hopf transform. Indeed, for $\gamma(t) = m_i$ on the interval $[q_i,q_{i+1})$, one can check that the function    
\begin{equation}\label{eq:cole_hopf}
\Phi_{\gamma}(t,x) = \frac{1}{m_i} \log \E\big[e^{m_i \Phi_{\gamma}(q_{i+1}, x+ \sqrt{q_{i+1}-t}Z)}\big],
\end{equation}
where $Z \sim \normal(0,1)$, is a solution on $[q_i,q_{i+1})$. Since $u_0 \le 0$ and $u_0 \in C^{\infty}(\R)$, we can check by induction that $\Phi_{\gamma}(t,x) \le 0$ for all $t,x$, so the above expectation is finite, and that $\Phi_{\gamma} \in C^{\infty}([q_i,q_{i+1}) \times \R)$ for all $i$. The next lemma provides classical and very useful estimates on the inverse Mill's ratio.  

\begin{lemma}\label{lem:A}
Let $\mA(x) = e^{-x^2/2}/(\sqrt{2\pi} \mN(x))$ be the inverse Mill's ratio. We have the following for all $x \in \R$:
\begin{itemize}
\item[(i)] $x \le \mA(x)$.
\item[(ii)] $x\mA(x)\le 1 + x^2$.
\item[(iii)] $\mA'(x) = \mA^2(x)-x\mA(x)$.
\item[(iv)] $0 \le \mA'(x) \le 1$. 
\item[(v)] For all $k \ge 3$, $\|\mA^{(k)}\|_{L^{\infty}(\R)} <\infty$.
\end{itemize}
\end{lemma}

\begin{proof}

Items $(i),(ii)$ are standard. They can be obtained through intergration by parts from the observation $\int_x^{\infty} (t-x)^{k}e^{-t^2/2} \rmd t \ge 0$ for $k=1$ and $k=2$ respectively; see \cite{gordon1941values} or~\cite{talagrand2011mean2}. 
Item $(iii)$ is immediate from the quotient rule. The lower bound in item $(iv)$ follows from the previous statements. The upper bound, following  \cite{sampford1953some}, is elegantly obtained from the observation $\var(Z | Z \ge x) \ge 0$ where $Z \sim N(0,1)$.
For item $(v)$, we first note that by iterating item $(iii)$ we obtain inductively that $\mA^{(k)}(x)=\mA(x)\cdot P(x,\mA(x))$ for some polynomial $P$. Therefore it suffices to control $|\mA^{(k)}(x)|$ for $|x|$ sufficiently large. The bound on $\R_+$ is non-trivial and follows from Corollary 1.6 of \cite{pinelis2019exact}. For $\R_-$ it holds as $\mA(x)$ is exponentially small for $x\to-\infty$, this means we have $\lim_{x\to-\infty} \mA^{(k)}=0$. 

\end{proof}

\paragraph{Proof of~\eqref{eq:phi}.} Let $\gamma = \sum_{i=0}^n m_i\indi_{[q_{i},q_{i+1})} \in \SF_+$. We will proceed by induction. We have 
\begin{align*}
\Phi_{\gamma}(q_n,x) &= \log \mN\Big(\frac{\kappa - x}{\sqrt{1-\ubq}}\Big)\\
&= -\frac{1}{2}\frac{(\kappa - x)^2}{1-\ubq} - \log \mA\Big(\frac{\kappa - x}{\sqrt{1-\ubq}}\Big).
\end{align*}
If $\frac{\kappa - x}{\sqrt{1-\ubq}} \le 1$ then $\mA\big(\frac{\kappa - x}{\sqrt{1-\ubq}}\big) \le \frac{1}{\int_1^{\infty}e^{-z^2/2}\rmd z} :=A$,
and 
\[\Phi_{\gamma}(q_n,x) \ge  -\frac{1}{2}\frac{(\kappa - x)^2}{1-\ubq} - \log A.\]
If $\frac{\kappa - x}{\sqrt{1-\ubq}} \ge 1$  then $\mA\big(\frac{\kappa - x}{\sqrt{1-\ubq}}\big) \le 1+ \frac{\kappa - x}{\sqrt{1-\ubq}}$, and 
\begin{align*}
\Phi_{\gamma}(q_n,x) &\ge  -\frac{1}{2}\frac{(\kappa - x)^2}{1-\ubq} - \log \Big(1+ \frac{\kappa - x}{\sqrt{1-\ubq}}\Big)\\
&\ge  -\frac{3}{2}\frac{(\kappa - x)^2}{1-\ubq} -\log A.
\end{align*}
Assume now that for $i \le n-1$,
\[\Phi_{\gamma}(q_{i+1},x) \ge -\frac{a_i}{2}(\kappa - x)^2 - b_i.\] 
Then using formula~\eqref{eq:cole_hopf}, we have for $t \in [q_i,q_{i+1})$,
\begin{align*}
\Phi_{\gamma}(t,x) &\ge  \frac{1}{m_i} \log \E\big[e^{- \frac{m_i a_i}{2} (\kappa - x - \sqrt{q_{i+1}-t}Z)^2 - m_ib_i}\big]\\
&=  -\frac{1}{2}\frac{a_i(\kappa - x)^2}{1+a_im_i(q_{i+1}-t)} - \frac{1}{2m_i}\log \Big(1+a_im_i(q_{i+1}-t)\Big) - b_i\\
&\ge -\frac{a_i}{2}(\kappa - x)^2 - \frac{a_i(q_{i+1}-q_i)}{2} - b_i.
\end{align*}
So we have $\Phi_{\gamma}(t,x) \ge \frac{a}{2}(\kappa - x)^2 - b$ for all $(t,x) \in [0,\ubq]\times \R$, with 
 $a=\frac{3}{1-\ubq}$ and $b = b_0 = \sum_{i} a_i(q_{i+1}-q_i)/2 -\log A= \frac{3\ubq}{2(1-\ubq)} - \log A.$

\paragraph{Proof of~\eqref{eq:dphi}.} Let $\gamma = \sum_{i=0}^n m_i\indi_{[q_{i},q_{i+1})}  \in \SF_+$. We will proceed by induction.

We start with the lower bound. 
Using formula~\eqref{eq:cole_hopf} we have for $t \in [q_{i},q_{i+1})$,
\[\partial_x \Phi_{\gamma}(t,x) = \left\langle \partial_x \Phi_{\gamma}(q_{i+1},x + \sqrt{q_{i+1}-t}Z) \right\rangle_{i}, \] 
where we define
\[\langle f(Z) \rangle_i := \frac{\E\big[f(Z)e^{m_i \Phi_{\gamma}(q_{i+1}, x+ \sqrt{q_{i+1}-t}Z)}\big]}{\E\big[e^{m_i \Phi_{\gamma}(q_{i+1}, x+ \sqrt{q_{i+1}-t}Z)}\big]}.\]
Since $\mA(x)>0$ we see by induction that $\partial_x \Phi_{\gamma}(t,x)>0$ for all $t,x$. Next, since $\mA(x) \ge x$, we have
\[\partial_x \Phi_{\gamma}(q_n,x) = \frac{1}{\sqrt{1-\ubq}} \mA\Big(\frac{\kappa - x}{\sqrt{1-\ubq}}\Big) \ge \frac{\kappa - x}{1-\ubq} = \frac{\kappa - x}{\int_{q_n}^1 \gamma(q)\rmd q}.\]
(Recall that $\ubq= q_n$.) Assume that the above inequality holds for $t=q_{i+1}$. 
Then for $t \in [q_{i},q_{i+1})$, we obtain by Gaussian integration by parts
\begin{align*}
\partial_x \Phi_{\gamma}(t,x) &\ge \frac{\kappa - x -  \sqrt{q_{i+1}-t} \langle Z\rangle_i}{\int_{q_{i+1}}^1 \gamma(q)\rmd q}\\
&= \frac{\kappa - x -  m_i(q_{i+1}-t) \partial_x \Phi_{\gamma}(t,x) }{\int_{q_{i+1}}^1 \gamma(q)\rmd q}.
\end{align*}
I.e., $\partial_x \Phi_{\gamma}(t,x)  \ge \frac{\kappa - x}{\int_{t}^1 \gamma(q)\rmd q}$. This concludes our proof of the lower bound. 

We now prove the upper bound.
Given the same estimates used to prove Eq.~\eqref{eq:phi}, we have
\[\partial_x \Phi_{\gamma}(q_n,x) = \frac{1}{\sqrt{1-\ubq}} \mA\Big(\frac{\kappa - x}{\sqrt{1-\ubq}}\Big)
\le\frac{1}{\sqrt{1-\ubq}} \cdot 
\begin{cases}
A &\mbox{if } \frac{\kappa - x}{\sqrt{1-\ubq}} \le 1,\\
1+ \frac{\kappa - x}{\sqrt{1-\ubq}} & \mbox{otherwise},
 \end{cases}
 \]
 with $A = \frac{1}{\int_1^{\infty}e^{-z^2/2}\rmd z}$.
Now assume that 
\[\partial_x \Phi_{\gamma}(q_{i+1},x)  \le 
\begin{cases}
A_i &\mbox{if } \frac{\kappa - x}{\sqrt{1-\ubq}} \le 1,\\
B_i\Big(1+ \frac{\kappa - x}{\sqrt{1-\ubq}}\Big) & \mbox{otherwise},
 \end{cases}
 \]
 for two constants $A_i, B_i$ which depend only on $\ubq$.
Let $v(Z) =  \frac{\kappa - x - \sqrt{q_{i+1}-t}Z}{\sqrt{1-\ubq}}$. Using the induction hypothesis we get
\begin{align*}
\partial_x \Phi_{\gamma}(t,x) &\le A_n \big\langle \indi_{v(Z) \le 1} \big\rangle_i +  B_n\big(1 +  \big\langle v(Z) \indi_{v(Z) > 1} \big\rangle_i\big)\\
&= A_i \big\langle \indi_{v(Z) \le 1} \big\rangle_i + B_i\Big(1 +  \frac{\kappa-x}{\sqrt{1-\ubq}} \big\langle \indi_{v(Z) > 1} \big\rangle_i\Big) 
- B_i\sqrt{\frac{q_{i+1}-t}{1-\ubq}} \big\langle Z \indi_{v(Z) > 1}\big\rangle_i.
\end{align*}
Let $w(x) = \frac{\kappa-x-\sqrt{1-\ubq}}{\sqrt{q_{i+1}-t}}$ and  $\Sigma = \E\big[e^{m_i \Phi_{\gamma}(q_{i+1}, x+ \sqrt{q_{i+1}-t}Z)}\big]$. Using integration by parts we have
\begin{align*}
\big\langle Z\indi_{v(Z) > 1} \big\rangle_i &= \big\langle Z \indi_{Z < w(x)} \big\rangle_i\\
&= \frac{1}{\Sigma} \int_{-\infty}^{w(x)} z e^{-z^2/2} e^{m_i \Phi_{\gamma}(q_{i+1}, x+ \sqrt{q_{i+1}-t}Z)} \frac{\rmd z}{\sqrt{2\pi}} \\
&= \frac{1}{\Sigma} \Big(- \frac{e^{w(x)^2/2}}{\sqrt{2\pi}} e^{m_i \Phi_{\gamma}(q_{i+1}, x - \sqrt{1-\ubq})} \\
&~~~+ m_{i} \sqrt{q_{i+1}-t}\int_{-\infty}^{w(x)} e^{-z^2/2}  \partial_x\Phi_{\gamma}(q_{i+1}, x+ \sqrt{q_{i+1}-t}Z) e^{m_i \Phi_{\gamma}(q_{i+1}, x+ \sqrt{q_{i+1}-t}Z)} \frac{\rmd z}{\sqrt{2\pi}}\Big)\\
& =- \frac{e^{-w(x)^2/2} e^{m_i \Phi_{\gamma}(q_{i+1}, x - \sqrt{1-\ubq})}}{\sqrt{2\pi}\Sigma}  
+ m_i  \sqrt{q_{i+1}-t} \big\langle \partial_x\Phi_{\gamma}(q_{i+1}, x+ \sqrt{q_{i+1}-t}Z)  \indi_{v(Z) > 1} \big\rangle_i\\
&\ge - \frac{e^{-w(x)^2/2} e^{m_i \Phi_{\gamma}(q_{i+1}, x - \sqrt{1-\ubq})}}{\sqrt{2\pi}\Sigma}.
\end{align*}
The last inequality follows since $\partial_x\Phi_{\gamma} \ge 0$. Now we argue that the above term is lower-bounded by a constant independent of $x$. Using the lower bound in Eq~\eqref{eq:phi}, 
\[\Sigma \ge \E\big[ e^{-\frac{3m_i}{2}v(Z)^2 -c_0 m_i} \big] = \frac{e^{-c_0m_i}}{\sqrt{1+ \sigma_i}} e^{-\frac{3}{2} \frac{m_i}{1+ \sigma_i} \frac{(\kappa-x)^2}{1-\ubq}}, \] 
with $\sigma_i = 3m_i \frac{q_{i+1}-t}{1-\ubq} \le 3m_i  \frac{q_{i+1}}{1-\ubq}$.
Since $\Phi_{\gamma}\le 0$, we obtain the bound
\[\frac{e^{-w(x)^2/2} e^{m_i \Phi_{\gamma}(q_{i+1}, x - \sqrt{1-\ubq})}}{\sqrt{2\pi}\Sigma} \le C \exp\Big(\frac{3}{2} \frac{m_i}{1+ \sigma_i} \frac{(\kappa-x)^2}{1-\ubq} -  \frac{1}{2}w(x)^2\Big),\]
where $C = C(\ubq)$. (This is the case since $0 \le m_i\le 1$ and $0 \le q_i \le \ubq$.)
Now suffices to check that $3 \frac{m_i}{1+ \sigma_i} \frac{1}{1-\ubq}< \frac{1}{q_{i+1}-t}$ to ensure that the above expression is globally upper bounded by a constant. This is indeed the case since $\ubq<1$.
Putting things together, for $t \in [q_{i},q_{i+1})$ we have 
\begin{align*}
\partial_x \Phi_{\gamma}(t,x) &\le A_i \big\langle \indi_{v(Z) \le 1} \big\rangle_i + B_i\Big(1 +  \frac{\kappa-x}{\sqrt{1-\ubq}} \big\langle \indi_{v(Z) > 1} \big\rangle_i\Big) + C_i\\
&\le \tilde{C}_i\Big(1 +  \frac{\kappa-x}{\sqrt{1-\ubq}} \vee 1\Big).
\end{align*}
where $\tilde{C}_i$ depends on $A_i, B_i$ and $C_i$, hence only on $\ubq$. This concludes our proof of the upper bound.

\paragraph{Proof of~\eqref{eq:ddphi}.} Let $\gamma \in \SF_+$. We proceed as usual by induction. We start with the lower bound. Using item $(iv)$ of Lemma~\ref{lem:A} we have
\[\partial_x^2 \Phi_{\gamma}(q_n,x) = - \frac{1}{1-\ubq} \mA'\Big(\frac{\kappa - x}{\sqrt{1-\ubq}}\Big) \ge -  \frac{1}{1-\ubq}.\]
Now, using formula~\eqref{eq:cole_hopf}, we have for $t \in [q_{i},q_{i+1})$,
\begin{align}\label{eq:second_derivative}
\partial_x^2 \Phi_{\gamma}(t,x) &= \left\langle \partial_x^2 \Phi_{\gamma}(q_{i+1},x + \sqrt{q_{i+1}-t}Z) \right\rangle_{i} \\
&~~+ m_i \Big(\left\langle \partial_x \Phi_{\gamma}(q_{i+1},x + \sqrt{q_{i+1}-t}Z)^2 \right\rangle_{i}-\left\langle \partial_x \Phi_{\gamma}(q_{i+1},x + \sqrt{q_{i+1}-t}Z) \right\rangle_{i}^2\Big).\nonumber
\end{align}
Since the expression in the last line is non-negative, the lower bound obtained at $t=q_n$ holds at all times.  
We now deal with the more computation-heavy upper bound. 
Without loss of generality it is enough to show that $\partial_x^2\Phi_{\gamma}(q_i,x) \le 0$ for all $i$.  
We apply formula~\eqref{eq:second_derivative} recursively to obtain for $j \le n$,
\begin{align}\label{eq:second_derivative_unfolded}
\partial_x^2 \Phi_{\gamma}(q_{j-1},x) &= \omega_{j}^n\Big( \partial_x^2 \Phi_{\gamma}\big(q_{n},x + \sum_{k=j}^n  \sqrt{q_{k}-q_{k-1}}Z_k\big) \Big) \\
&~~+ \sum_{i=j}^n m_{i-1} \Biggr(\omega_{j}^i\Big( \partial_x \Phi_{\gamma}\big(q_{i},x + \sum_{k=j}^{i} \sqrt{q_{k}-q_{k-1}}Z_k)^2 \Big) \nonumber\\
&\hspace{3cm}- \omega_{j}^{i-1}\Big( \partial_x \Phi_{\gamma}\big(q_{i-1}, x + \sum_{k=j}^{i-1} \sqrt{q_{k}-q_{k-1}}Z_k \big)^2\Big) \Biggr).\nonumber
\end{align}
In the above, $Z_1,\cdots,Z_n$ are i.i.d.\ $N(0,1)$ random variables. Next we define the symbols $\omega_{j}^i$. Let
\[\Sigma_{j}^i := \sum_{k=j}^{i} \sqrt{q_{k}-q_{k-1}}Z_k.\] 
Then for $j \le i$,
\[\omega_{j}^i\big(f\big((Z_k)_j^i\big)\big) := \frac{\E\big[\omega_{j+1}^{i}\big(f\big(Z_j, (Z_k)_{j+1}^i \big)\big) e^{m_{j-1} \Phi_{\gamma}(q_{j}, x+\Sigma_{j}^{i})} \, \big| \, Z_{j+1},\cdots,Z_i \big]}{\E\big[e^{m_{j-1} \Phi_{\gamma}(q_{j}, x+ \Sigma_{j}^{i})} \, \big| \, Z_{j+1},\cdots,Z_i  \big]}.\] 
$\omega_{i+1}^{i} \equiv \mbox{id}$. (Note that $\omega_{i}^{i}(\cdot) = \langle \cdot \rangle_j$.)
A substantial part of the computation is to bound the first expression in~\eqref{eq:second_derivative_unfolded}.  
For ease of notation we let $V_{j}^n = \frac{\kappa - x - \Sigma_j^n}{\sqrt{1-\ubq}}$. Using item $(iii)$ of Lemma~\ref{lem:A},
\begin{align}
\partial_x^2 \Phi_{\gamma}(q_{n},x + \Sigma_j^n)  &= - \frac{1}{1-\ubq} \mA'(V_{j}^n) \label{eq:phi2} \\
&= \frac{1}{1-\ubq} \left(V_{j}^n \mA(V_{j}^n) -\mA(V_{j}^n)^2 \right),\nonumber
\end{align}
and
\begin{equation}\label{eq:v_A}
\omega_{j}^n\big( V_{j}^n  \mA(V_{j}^n) \big)
= \frac{\kappa - x}{\sqrt{1-\ubq}} \, \omega_{j}^n\left(\mA(V_{j}^n) \right) - \sum_{k=j}^n \sqrt{\frac{q_{k}-q_{k-1}}{1-\ubq}} \, \omega_{j}^n\left( Z_k \cdot \mA(V_{j}^n) \right).
\end{equation}
Furthermore, by the rule of iterated expectations and then Gaussian integration by parts,  
\begin{align}\label{eq:z_omega}
\omega_{j}^n\left( Z_k \cdot \mA(V_{j}^n) \right) 
&= \omega_{j}^{k} \left( Z_k \cdot \omega_{k+1}^n\left(\mA(V_{j}^n) \right)\right)\nonumber\\
&= \sqrt{q_k - q_{k-1}} \, \omega_{j}^{k} \left( \frac{\rmd}{\rmd x}  \omega_{k+1}^n\big(\mA(V_{j}^n)\big)\right)\\
&~~~+ m_{k-1}\sqrt{q_{k} - q_{k-1}} \, \omega_{j}^{k} \left( \partial_x \Phi_{\gamma}(q_{k}, x+ \Sigma_{j}^{k}) \cdot \omega_{k+1}^n\left(\mA(V_{j}^n) \right)\right).\nonumber
\end{align}
On the one hand, 
\begin{align*}
\omega_{j}^{k} \left( \partial_x \Phi_{\gamma}(q_{k}, x+ \Sigma_{j}^{k}) \cdot \omega_{k+1}^n\left(\mA(V_{j}^n) \right)\right)
&=\sqrt{1-\ubq}  \, \omega_{j}^{k} \left( \partial_x \Phi_{\gamma}(q_{k}, x+ \Sigma_{j}^{k}) \cdot \omega_{k+1}^n\left(\partial_x \Phi_{\gamma}(q_{n}, x+ \Sigma_{j}^{n}) \right)\right)\\
&= \sqrt{1-\ubq} \, \omega_{j}^{k} \left( \partial_x \Phi_{\gamma}(q_{k}, x+ \Sigma_{j}^{k})^2\right).
\end{align*} 
On the other hand we have the recursion
\begin{align*} 
 \frac{\rmd}{\rmd x} \omega_{k+1}^n\big(\mA(V_{j}^n)\big) &=  \frac{\rmd}{\rmd x} \frac{\E\big[\omega_{k+2}^{n}\big(\mA(V_{j}^n)\big) e^{m_k \Phi_{\gamma}(q_{k+1}, x+\Sigma_{j}^{k+1})} \, \big| \, Z_{k+2},\cdots,Z_n \big]}{\E\big[e^{m_k \Phi_{\gamma}(q_{k+1}, x+ \Sigma_{j}^{k+1})} \, \big| \, Z_{k+2},\cdots,Z_n  \big]}\\
&= \Big\langle  \frac{\rmd}{\rmd x}  \omega_{k+2}^n\big(\mA(V_{j}^n)\big)  \Big\rangle_{k+1} \\
&~~~ + m_{k} \big\langle  \omega_{k+2}^n\big(\mA(V_{j}^n)\big) \partial_x \Phi_{\gamma}(q_{k+1}, x+ \Sigma_{j}^{k+1}) \big\rangle_{k+1} \\
&~~~ -  m_{k} \big\langle  \omega_{k+2}^n\big(\mA(V_{j}^n)\big) \big\rangle_{k+1} \cdot \big \langle \partial_x \Phi_{\gamma}(q_{k+1}, x+ \Sigma_{j}^{k+1})\big\rangle_{k+1}\\
&= \Big\langle  \frac{\rmd}{\rmd x}  \omega_{k+2}^n\big(\mA(V_{j}^n)\big)  \Big\rangle_{k+1} \\
&~~~ + m_{k}\sqrt{1-\ubq} \Big( \big\langle \partial_x \Phi_{\gamma}(q_{k+1}, x+ \Sigma_{j}^{k+1})^2 \big\rangle_{k+1} 
-  \partial_x \Phi_{\gamma}(q_{k}, x+ \Sigma_{j}^{k})^2 \Big).
\end{align*}
We write an explicit formula for the above recursion:  
\begin{align*} 
 \frac{\rmd}{\rmd x} \omega_{k+1}^n\big(\mA(V_{j}^n)\big) &= - \frac{1}{\sqrt{1-\ubq}} \, \omega_{k+1}^n\big(\mA'(V_{j}^n)\big) \\
 &~~~+ \sqrt{1-\ubq} \sum_{l=k+1}^n m_{l-1} \Big( \omega_{k+1}^{l} \big(\partial_x \Phi_{\gamma}(q_{l}, x+ \Sigma_{j}^{l})^2\big)
-  \omega_{k+1}^{l-1} \big(\partial_x \Phi_{\gamma}(q_{l-1}, x+ \Sigma_{j}^{l-1})^2 \big) \Big).
\end{align*} 
Plugging this into Eq.~\eqref{eq:z_omega} we obtain
\begin{align*} 
\omega_{j}^n\left( Z_k \cdot \mA(V_{j}^n) \right) &= -\sqrt{\frac{q_k - q_{k-1}}{1-\ubq}} \,   \omega_{j}^n\big(\mA'(V_{j}^n)\big)\\
&+ \sqrt{(q_k - q_{k-1})(1-\ubq)} \sum_{l=k+1}^n m_{l-1} \Big( \omega_{j}^{l} \big(\partial_x \Phi_{\gamma}(q_{l}, x+ \Sigma_{j}^{l})^2\big)
-  \omega_{j}^{l-1} \big(\partial_x \Phi_{\gamma}(q_{l-1}, x+ \Sigma_{j}^{l-1})^2 \big) \Big)\\
&+  \sqrt{(q_k - q_{k-1})(1-\ubq)} \, m_{k-1}  \omega_{j}^{k} \left( \partial_x \Phi_{\gamma}(q_{k}, x+ \Sigma_{j}^{k})^2\right).
\end{align*} 
 Pugging this into Eq.~\eqref{eq:v_A} we obtain 
 \begin{align*}
 \omega_{j}^n\big( \mA(V_{j}^n)' \big) &= \omega_{j}^n\big( \mA(V_{j}^n)^2 \big) - \omega_{j}^n\big( V_{j}^n  \mA(V_{j}^n) \big)\\
 &= \omega_{j}^n\big( \mA(V_{j}^n)^2 \big) - \frac{\kappa - x}{\sqrt{1-\ubq}} \, \omega_{j}^n\left(\mA(V_{j}^n) \right) 
+ \sum_{k=j}^n \sqrt{\frac{q_{k}-q_{k-1}}{1-\ubq}} \, \omega_{j}^n\left( Z_k \cdot \mA(V_{j}^n) \right)\\
&= \omega_{j}^n\big( \mA(V_{j}^n)^2 \big) - \frac{\kappa - x}{\sqrt{1-\ubq}} \, \omega_{j}^n\left(\mA(V_{j}^n) \right) \\
 &~~~ - \sum_{k=j}^n \frac{q_{k}-q_{k-1}}{1-\ubq} \,  \omega_{j}^n\big( \mA(V_{j}^n)' \big) \\
 &~~~ + \sum_{k=j}^n (q_{k}-q_{k-1}) 
  \sum_{l=k+1}^n m_{l-1} \Big( \omega_{j}^{l} \big(\partial_x \Phi_{\gamma}(q_{l}, x+ \Sigma_{j}^{l})^2\big)
-  \omega_{j}^{l-1} \big(\partial_x \Phi_{\gamma}(q_{l-1}, x+ \Sigma_{j}^{l-1})^2 \big) \Big)\\
 &~~~+\sum_{k=j}^n (q_{k}-q_{k-1}) m_{k-1} \,  \omega_{j}^{k} \left( \partial_x \Phi_{\gamma}(q_{k}, x+ \Sigma_{j}^{k})^2\right).
 \end{align*}
 Hence, we obtain an expression for $ \omega_{j}^n\big( \mA(V_{j}^n)' \big)$, and therefore also for $ \omega_{j}^n\big(\partial_x^2 \Phi_{\gamma}(q_{n},x + \Sigma_j^n) \big)$ via the relation~\eqref{eq:phi2}:
  \begin{align}\label{eq:omega_Ap}
 \frac{1-q_{j-1}}{1-\ubq} \, &\omega_{j}^n\big( \mA(V_{j}^n)' \big) 
= \omega_{j}^n\big( \mA(V_{j}^n)^2 \big) - \frac{\kappa - x}{\sqrt{1-\ubq}} \, \omega_{j}^n\left(\mA(V_{j}^n) \right) \\
 & + \sum_{l=j+1}^n (q_{l-1}-q_{j-1}) 
   m_{l-1} \Big( \omega_{j}^{l} \big(\partial_x \Phi_{\gamma}(q_{l}, x+ \Sigma_{j}^{l})^2\big)
-  \omega_{j}^{l-1} \big(\partial_x \Phi_{\gamma}(q_{l-1}, x+ \Sigma_{j}^{l-1})^2 \big) \Big)\nonumber \\
&+\sum_{k=j}^n (q_{k}-q_{k-1}) m_{k-1} \,  \omega_{j}^{k} \left( \partial_x \Phi_{\gamma}(q_{k}, x+ \Sigma_{j}^{k})^2\right).\nonumber
 \end{align}
 At this stage we want to eliminate the explicit appearance of $\kappa - x$ in the above expression. For this we use the lower bound~\eqref{eq:dphi} with $t=q_{j-1}$:
 \[\frac{\kappa - x}{\int_{q_{j-1}}^1 \gamma(q)\rmd q} \le \partial_{x} \Phi_{\gamma}(q_{j-1},x).\]
(This will be the only inequality used in this argument.) Further, we also have 
\[\frac{1}{\sqrt{1-\ubq}} \, \omega_{j}^n\left(\mA(V_{j}^n) \right) = \partial_{x} \Phi_{\gamma}(q_{j-1},x),\]
 so we obtain a lower bound for the right-hand side of Eq.~\eqref{eq:omega_Ap}:
 \begin{align*}
\omega_{j}^n\big( \mA(V_{j}^n)^2 \big) -  \frac{\kappa - x}{\sqrt{1-\ubq}} \, \omega_{j}^n\big( \mA(V_{j}^n) \big) 
&\ge  \omega_{j}^n\big( \mA(V_{j}^n)^2 \big) - \Big(\int_{q_{j-1}}^1 \gamma(q)\rmd q\Big) \partial_{x} \Phi_{\gamma}(q_{j-1},x)^2 \\ 
&= (1-\ubq) \Big(\omega_{j}^{n} \big(\partial_x \Phi_{\gamma}(q_{n}, x+ \Sigma_{j}^{n})^2\big)
-   \partial_{x} \Phi_{\gamma}(q_{j-1},x)^2\Big)\\
&~~~ - \sum_{k=j}^n  (q_{k}-q_{k-1}) m_{k-1}  \partial_{x} \Phi_{\gamma}(q_{j-1},x)^2.
\end{align*}
Let us define
 \begin{align*}
 \Delta_j^{k} := \omega_{j}^{k} \big(\partial_x \Phi_{\gamma}(q_{k}, x+ \Sigma_{j}^{k})^2\big)
-  \omega_{j}^{k-1} \big(\partial_x \Phi_{\gamma}(q_{k-1}, x+ \Sigma_{j}^{k-1})^2\big) \ge 0.  
 \end{align*}
 At last, we put everything together in Eq.~\eqref{eq:second_derivative_unfolded} to obtain
 \begin{align*}
 \partial_x^2 \Phi_{\gamma}(q_{j-1},x) &\le  \sum_{k=j}^n \Big(m_{k-1}  - \frac{1-\ubq}{1-q_{j-1}} - \frac{q_{k-1}-q_{j-1}}{1-q_{j-1}} m_{k-1} \Big)\Delta_{j}^k \\
 & - \sum_{k=j}^n  \frac{q_{k}-q_{k-1}}{1-q_{j-1}}  m_{k-1} \,  \omega_{j}^{k} \left( \partial_x \Phi_{\gamma}(q_{k}, x+ \Sigma_{j}^{k})^2\right)
+   \sum_{k=j}^n\frac{q_{k}-q_{k-1}}{1-q_{j-1}} m_{k-1}  \partial_{x} \Phi_{\gamma}(q_{j-1},x)^2\\
 &= \sum_{k=j}^n \delta_k \Delta_k ,
 \end{align*}
with 
 \begin{align*}
 \delta_k &= m_{k-1}  - \frac{1-\ubq}{1-q_{j-1}} - \frac{q_{k-1}-q_{j-1}}{1-q_{j-1}} m_{k-1} - \sum_{l=k}^n \frac{q_{l}-q_{l-1}}{1-q_{j-1}}  m_{l-1}\\
 &= -\frac{1}{1-q_{j-1}}\sum_{l=k}^{n+1} (q_{l}-q_{l-1})(m_{l-1}-m_{k-1}).
 \end{align*}
 Wee readily see that $\delta_k \le 0$ when $m_1 \le m_2 \le \cdots \le m_n$, i.e., when $\gamma$ is non-decreasing.
 
\begin{remark}
The above bound can also be obtained (perhaps in a shorter and more transparent way) via stochastic calculus, provided one is allowed to use it.  Indeed, we do not apriori know that the SDE~\eqref{eq:SDE2} has a solution; that is, prior to proving that $\partial_x^2\Phi_{\gamma}$ is bounded. We opted instead for the above discrete-time approach.    
\end{remark}

\paragraph{Proof of~\eqref{eq:higher_dphi}.}
Let $\gamma \in \SF_+$. For $\psi_k := \partial_x^k\Phi_{\gamma}$ we have 
\begin{align*}
\partial_t\psi_k + \frac{1}{2}\gamma(t)\partial_x^k\big[\big(\partial_x \Phi_{\gamma}\big)^2\big] +  \frac{1}{2}\partial^{2}_{x}\psi_k &= 0, \\ \mbox{and}~~~\psi_k(\ubq,x) &= \frac{(-1)^k}{(1-\ubq)^{k/2}}u_0^{(k)}\Big(\frac{\kappa- x}{\sqrt{1-\ubq}}\Big).
\end{align*}
Given the bounds~\eqref{eq:dphi} and~\eqref{eq:ddphi}, we are in a position to use stochastic calculus. Consider $(X_t)$, the unique strong solution to the SDE $\rmd X_t = \gamma(t) \partial_x\Phi_{\gamma}(t,X_t) \rmd t + \rmd B_t$ with $X_{0}=0$. 
Since $\psi_k$ is $C^{\infty}([a,b]\times \R)$ on intervals $[a,b]$ of continuity of $\gamma$, we can use It\^o's formula to write
\begin{align*}
\rmd \psi_k(t,X_t) &= \Big(\partial_t\psi_k + \gamma\partial_x \Phi_{\gamma} \partial_x\psi_k +  \frac{1}{2}\partial^{2}_{x}\psi_k\Big)(t,X_t) \rmd t + \partial_x \psi_k (t,X_t) \rmd B_t\\
&= -\frac{\gamma(t)}{2}\sum_{j=1}^{k-1} {k\choose j} \psi_{j+1}(t,X_t) \psi_{k-j+1}(t,X_t) \rmd t +  \psi_{k+1} (t,X_t) \rmd B_t.
\end{align*}
Whence,
\begin{align*}
\psi_k(t,x) &= k\E_{t,x}\int_t^{\ubq} \gamma(s)\psi_{2}(s,X_s) \psi_{k}(s,X_s) \rmd s \\  
&~~~+\frac{(-1)^k}{(1-\ubq)^{k/2}} \E_{t,x} \left[u_0^{(k)}\Big(\frac{\kappa-X_{\ubq}}{\sqrt{1-\ubq}}\Big) \right] \\
&~~~+ \frac{1}{2}\sum_{j=2}^{k-2} {k\choose j} \E_{t,x}\int_t^{\ubq} \gamma(s)\psi_{j+1}(s,X_s) \psi_{k-j+1}(s,X_s) \rmd s.
\end{align*}
(The sum in the last line is vacuous for $k=3$.)

Let  $y_k(t) = \|\psi_k(t,\cdot)\|_{L^{\infty}(\R)}$. Since $|\psi_2| \le \frac{1}{1-\ubq}$ (Eq~\eqref{eq:ddphi}) and $0 \le \gamma \le 1$, we obtain for $k= 3$,
\begin{align*}
y_3(t) \le \frac{3}{1-\ubq} \int_t^{\ubq} y_{3}(s) \rmd s + \frac{1}{(1-\ubq)^{3/2}} \|u_0^{(3)}\|_{L^{\infty}(\R)}.
\end{align*}
So by Gronwall's lemma, $y(t) \le C_0 \|u_0^{(3)}\|_{L^{\infty}(\R)}$ for all $t \le\ubq$, where $C_0$ depends only on $\ubq$. Hence  $\|\psi_3\|_{L^{\infty}([0,\ubq]\times \R)} <\infty$.
Next, assume for the sake of induction that 
\[\|\psi_j\|_{L^{\infty}([0,\ubq]\times \R)} <\infty,~~~ \forall\,  3 \le j \le k-1.\]
Then we obtain via Gronwall's lemma 
\[\|\psi_k\|_{L^{\infty}([0,\ubq]\times \R)} \lesssim_{\ubq,k} \|u_0^{(k)}\|_{L^{\infty}(\R)} + \frac{1}{2}\sum_{j=2}^{k-2} {k\choose j} \|\psi_{j+1}\|_{L^{\infty}([0,\ubq]\times \R)} \cdot \|\psi_{k-j+1}\|_{L^{\infty}([0,\ubq]\times \R)} <\infty.\] 
(Here, $A \lesssim_{\ubq,k} B$ means $A \le C B$ where $C$ depends only on $\ubq$ and $k$.)

\section*{Acknowledgments} 
We are grateful to Pierfrancesco Urbani for instructive discussions about the jamming transition and its connection to the perceptron problem. We thank Yin Tat Lee for suggesting the use of the result in \cite{jiang2020improved} and Chris Jones for bringing~\cite{jones2020spherical} to our attention. Part of this work was done while the authors were visiting the Simons Institute for the Theory of Computing as part of the fall 2020 program on Probability, Geometry and Computation in High Dimensions.

\begin{small}

\bibliographystyle{alpha}
\bibliography{all-bib}

\end{small}

\end{document}